\documentclass{article}
\usepackage{ggw-arXiv}
\usepackage{amssymb}
\usepackage{amsmath}
\usepackage{pdflscape}
\usepackage{graphicx}
\usepackage{caption}
\usepackage{eufrak}

\theoremstyle{theorem}

\newtheorem{proposition}{Proposition}
\newtheorem{conj}{Conjecture}

\theoremstyle{definition}

\newtheorem*{remarks}{Remarks}

\begin{document}
\title{Universal Linear Mean Relationships in All Polynomials}
\markright{Universal Mean Relationships}
\author{Gregory Gerard Wojnar \\ Daniel Sz. Wojnar  \\ Leon Q. Brin}
 
\maketitle

\begin{abstract}
In any cubic polynomial, the average of the slopes at the $3$ roots is the
negation of the slope at the average of the roots.  In any quartic, the
average of the slopes at the $4$ roots is twice the negation of the slope at
the average of the roots.  We generalize such situations and present a
procedure for determining all such relationships for polynomials of any
degree.  E.g., in any septic $f$,  letting $\overline{f}_n$ denote the mean
$f$ value over all zeroes of the derivative $f^{\left(n\right)}$, it holds
that $37$ $\overline{f}_1-150$
$\overline{f}_3+200\,\overline{f}_4-135\,\overline{f}_5+48\,%
\overline{f}_6=0$; and in any quartic it holds that $5\, \overline{f}_1-6$
$\overline{f}_2+1\,\overline{f}_3=0$.  Having calculated such relationships in all dimensions up to 49, in all even dimensions there is a single relationship, in all odd dimensions there is a two-dimensional family of relationships.  We see connections to Tchebyshev, Bernoulli, \& Euler polynomials, and Stirling numbers.
\end{abstract}

This paper started with (a) the
observation that the quadratic formula essentially provides the roots to be
$x=\mu\pm\sigma\,$where $\mu$ is the mean of the two roots and where $\sigma$
is the standard deviation of the two roots considered equally likely, (b) the
observation that the slopes at the roots are
$\pm\sqrt{\textit{Discriminant}}=\pm 2\,a\sigma$, and (c) the trivial
observation that the average of the slopes at the roots equals the slope at
the average of the roots. \ Curiosity took us next to cubics, to find that
the Cardano-Tartaglia cubic formula can be perceived as providing the roots
to be
$r_{_k}=\mathbb{E}+\omega^k\overset{\rightharpoonup}{{\,T_{_{+}}}}
+\omega^{-\,k}\overset{\rightharpoonup}{{\,T_{_{-}}}}$ where
$\omega=-\frac{1}{2}+i\frac{\sqrt{3}\,}{2}$ is a primitive cube root of
unity, $k\in\left\{0,1,2\right\}$, and where
$\overset{\rightharpoonup}{{\,T_{_\pm}}}:=\sqrt[3]{\frac{%
\mathbb{W}}{2}\pm\sqrt{\left(\frac{\mathbb{W}}{2}\right)^2-\left(\frac{%
\mathbb{V}}{2}\right)^3}}\,$with $\mathbb{V}$ being the variance of the $3$
roots, $\mathbb{W}$ being the $3^{\text{rd}}$ central moment (sometimes
referred to as the \emph{unscaled skewness}), and $\mathbb{E}\equiv\mu$ being the
expectation of the roots, considered as being equally likely.  Moreover, the
coordinates of the cubic's inflection point are
$\left(\mathbb{E},\,-a\,\mathbb{W}\right)$, and the slope at the inflection
point is the negation of the mean slope at the $3$ roots.  Indeed, the
inflection point slope will be $-\frac{3}{2}\,a\mathbb{V}$. \ \ Next we saw
that the mean slope at the $4$ roots of a quartic is the negation of twice
the slope at the point $\left(\mathbb{E},\,f\left(\mathbb{E}\right)\right)$
$\ldots$ and the hunt was on to find all such relationships!  It was not
just mean slope relationships either-- e.g., one can of course note that the
mean of a cubic's function values at the roots of the first derivative equals
the `mean' of the function's (sole) value at the root of the second derivative, viz.
$-a\mathbb{W}$. \ What other relationships lie waiting to be discovered? Have we only seen the tip of an iceberg?
Perhaps the reader is starting to suspect that such relationships are really
typical.

\section{The General Question}

We establish some notations.  For any degree $D$ polynomial $f:\mathbb{C}\to\mathbb{C}$ let
$\mathcal{R}\equiv\mathcal{R}_{_f}$ denote the family ($\equiv$ multiset) of $D$ roots, and let
$\mathcal{R}^{\left(\rho\right)}\equiv\mathcal{R}^{\left(\rho\right)}_{_f}$
denote the family of $\left(D-\rho\right)\,$roots of the $\rho^{\text{th}}$
derivative $f^{\left(\rho\right)}$  (sometimes using primemarks
$\left[^{\,\,\prime}\right]$ for lower orders); we also write $|\mathcal{R}_f|=D$ (even if some root(s) have multiplicity greater than $1$).  Denote the average value of
the function on the $\rho^{\text{th}}$ derivative roots family as
$\overline{{f\left(\mathcal{R}^{\left(\rho\right)}\right)}}$, and similarly
use 
$\overline{{f^{\left(\delta\right)}\left(\mathcal{R}^{\left(\rho\right)}
\right)}}$ for the average value of the $\delta^{\text{ th}}$ derivative of
$f$ over the same roots family.  Sometimes we will emphasize the degree of
$f$ by explicitly writing $f_{_D}$, and we put $\varphi_{_{D\delta\rho}}:=$
$\overline{{f_{_D}^{\left(\delta\right)}\left(\mathcal{R}^{\left(\rho\right)}
\right)}}$ .

Our initial results mentioned above suggest that we seek to determine 
$\varphi_{_{D\delta\rho}}$ for all $D,\,\delta,\,$\& $\rho$, and that we seek
relationships among various $\varphi_{_{D\delta\rho}}$, of the form
$\sum_\rho\alpha_{_\rho}\,\varphi_{_{D\delta\rho}}$ $=0$.  (Other linear
combinations of the $\varphi_{_{D\delta\rho}}$ that sum over $D$ or over
$\delta$ are unnatural because of dimensionality considerations.) \\[1\baselineskip]
\noindent
\emph{Proof Routes}\\[-0\baselineskip]
\noindent
In cases where the cardinality of the roots family
$\mathcal{R}^{\left(\rho\right)}$ is $2$ or $3$, we can compute by brute
force since the values of the roots will be given by either the quadratic or
cubic formula.  Perhaps this method could even be pushed to the case of a
family of $4$ roots via the Ferarri quartic formula, but the level of
complexity is greatly increased.  In any case, for root families of size $5$
or greater, we need some alternative route. \\

Again we benefit from establishing notations.  We shall use what we refer to
as the quasi-binomial representation of polynomials, expressing coefficients
in terms of averaged symmetric polynomials in the roots.  For example, a
general cubic is represented as
$f\left(x\right)=a\left(x^3-3\,\overline{r}\,x^2+3\,\overline{\overline{r}}\,
x-\overline{\overline{\overline{r}}}\right)$ where it turns out that
$\overline{r}$ is the average of the $\binom{3}{1}$ roots, $\overline{\overline{r}}$ is
the average of the $\binom{3}{2}\,$products of pairs of roots, and
$\overline{\overline{\overline{r}}}\,$ is the `average' of the $\binom{3}{3}$
product of triplets of roots. \ With these, the earlier-mentioned
Cardano-Tartaglia cubic formula can be computed easily via
$\mathbb{E}=\overline{r},\,\mathbb{V}=2\left(\overline{r}^{\,2}-\overline{%
\overline{r}}\right),\,$and
$\mathbb{W}=2$~$\overline{r}^{\,3}-3\,\overline{r}\,\overline{\overline{r}}\,%
+\overline{\overline{\overline{r}}}$. In general we write 
$f_D(x)=\sum\limits_{\substack{i\,+\,j\,\,=\,D\\0\,\leq\,i,\,j\,\,\leq\,D}}\left(-1%
\right)^i \binom{D}{i\quad j} \,\overline{\,\,\,r}\!\!\!\!\!^{^i}\,\,\,\, x^j $  where
$\,\overline{\,\,\,r}\!\!\!\!\!^{^i}\,\,\,\,$ denotes that quasi-binomial
parameter with $i$-many bars, $\binom{D}{i\quad j}=\frac{D!}{i!\,j!}$, and we write
$\overset{\rightharpoonup}{\mathcal{R}}$ for the \emph{ordered} family $\left( \overline{r},\ \overline{\overline{r}}, \ \overline{\overline{\overline{r}}},\ \ldots\right)$. Notice that
$f^{\,\prime}\left(x\right)=3\cdot
a\left(x^2-2\,\overline{r}\,x+\overline{\overline{r}}\right)\,$is just $3$
times the quasi-binomial representation of a generic quadratic.  This is
typical:  with quasi-binomial representations of polynomials,
$f^{\,\prime}_{_D}=D\cdot f_{_{D\,-\,1}}$, i.e. taking the derivative coincides with truncating the highest order term from the quasi-binomial family $\overset{\rightharpoonup}{\mathcal{R}}$; in other words, for a given quasi-binomial parameter family $\overset{\rightharpoonup}{\mathcal{R}}$, we obtain a finite Appell sequence of derived functions.  Also, in terms of the
quasi-binomial parameters, the statistically-presented version of the
quadratic formula is simply
$x=\overline{r}\pm\sqrt{\overline{r}^{\,\,2}-\overline{\overline{r}}\,}$. \\

It remains instructive to establish that the slope of a cubic at its
inflection point is the negated mean of the slopes at the three roots, via the
route of brute force computation.  We want to evaluate $f'$ at the
three roots
$r_{_k}=\mathbb{E}+\omega^k\overset{\rightharpoonup}{{\,T_{_{+}}}}+
\omega^{-\,k}\overset{\rightharpoonup}{{\,T_{_{-}}
}}$(for~$k\in\left\{0,\,1,\,2\right\}$): \ \ $f'\left(r_{_k}\right)=3\cdot
a\left(r_{_k}^2-2\,\overline{r}\,r_{_k}+\overline{\overline{r}}\right)$. \
For this we will use \\[-0.75\baselineskip]
\begin{multline*}
r_{_k}^2= \left(\mathbb{E}+\omega^k\overset{\rightharpoonup}{{\,
T_{_{+}}}}\ +\omega^{-\,k}\overset{\rightharpoonup}{{\,T_{_{-}}}}\right)^2= \\ 
\overline{r}\,^2+\omega^{2\,k}\overset{\rightharpoonup}{{\,
T_{_{+}}}}^2+\omega^{-2\,\,k}\overset{\rightharpoonup}{{\,
T_{_{-}}}}^2+2\left(\overline{r}\left(\omega^k\overset{
\rightharpoonup}{{\,T_{_{+}}}}\, +\omega^{-\,k}\overset{\rightharpoonup}{{\,T_{_{-}}}}\right)+
\omega^k\overset{\rightharpoonup}{{\,T_{_{+}}}}\,\cdot\omega^{-\,k}\overset{\rightharpoonup}{{\,T_{_{-}}}}\right)\text{.}
\end{multline*}
Now averaging over $k\in\left\{0,1,2\right\}$ we encounter convenient terms such
as \linebreak $\left(\omega^{2\cdot 0}+\omega^{2\cdot 1}+ \omega^{2\cdot
2}\right)\overset{\rightharpoonup}{{\,T_{_{+}}}}^2\,=\,\left(
\omega^0+\omega^2+\omega^1\right)\overset{\rightharpoonup}{{\,
T_{_{+}}}}^2=0$, and $\left(\,\omega^0+\omega^1+\omega^2\right)\overset{\rightharpoonup}{{\,
T_{_{+}}}}=0$, and three copies of $1\cdot
\overset{\rightharpoonup}{{\,T_{_{+}}}}
\overset{\rightharpoonup}{{\,T_{_{-}}}}$.  Hence we obtain 
$\overline{r_{_k}^2}=\,\overline{r}\,^2+2\,
\overset{\rightharpoonup}{{\,T_{_{+}}}}
\overset{\rightharpoonup}{{\,T_{_{-}}}}$.  From the definitions, we
further simplify \ $\overset{\rightharpoonup}{{2\,T_{_{+}}}}
\overset{\rightharpoonup}{{\,T_{_{-}}}}=\mathbb{V}$. Of course $\,\overline{r_{_k}}=\mathbb{E}=\overline{r}$.  Together
we now have $\overline{f'\left(r_{_k}\right)}=3\cdot
a\left(\mathbb{E}^2+\mathbb{V}-2\,\overline{r}\,\overline{r}+\overline{
\overline{r}}\right)=3\cdot
a\left(-\,\overline{r}\,^2+2\left(\overline{r}^{\,\,2}-\overline{
\overline{r}}\,\right)+\overline{\overline{r}}\right)\,=$ $3\cdot
a\left(\,\overline{r}\,^2-\overline{\overline{r}}\,\right)=\frac{3}{2}\,a
\mathbb{V}$.  Whereas the inflection point occurs at $x=\overline{r}$, we
also compute $f'\left(\overline{r}\right)=3\cdot
a\left(\,\overline{r}\,^2-2\,\overline{r}\,\,\overline{r}\,+\overline{
\overline{r}}\right)=-\frac{3}{2}\,a\mathbb{V}$, thus establishing our goal. \\

Looking back, our efforts were aided here by the fact that
$\frac{\omega^0+\omega^1+\omega^2}{3}=0$, something particular to the case of
cube roots-- something that we won't have available for general cases.  What
we \emph{will} have in the general case is computing means of powers of the roots,
and here is our key to higher degree proofs:  we want to be able to express
powers of the roots in terms of the quasi-binomial parameters,  i.e. in
terms of the mean elementary symmetric polynomials in the roots. 
Exemplifying this in the case of $3$ roots, consider: \\[-1\baselineskip]
$$\left(\frac{1}{3}\sum r_{_k}\right)^2=\frac{1}{3}\left\{\frac{1}{3}\sum
r_{_k}^{\,2}+\frac{2}{3}\sum r_{_k}r_{_{k'}}\right\}\,\text{ whence}$$ 

\begin{equation}\label{eqn:rsquaredbar}
\overline{r_{_k}^{\,2}}\equiv \frac{1}{3}\sum r_{_k}^{\,2}=
3\,\overline{r}^{\,2}-2\,\overline{\overline{r}}\, \
=\mathbb{V}+\mathbb{E}^2\text{ .}
\end{equation}
\noindent
For any (multi)set $\mathcal{Z}$ of numbers, the elementary symmetric polynomials on $\mathcal{Z}$ are defined as $e_i(\mathcal{Z})\equiv e_i := \sum\limits_{\substack{A\subseteq\mathcal{Z} \\ |A|=i}}\  \prod\limits_{z \in A} z$ \quad (for $i=0\ldots|\mathcal{Z}|; e_0=1$);
 e.g. if $\mathcal{Z}=\mathcal{R}$ as above with $|\mathcal{R}|=3$, then $e_{_2}$ is the sum of all $\binom{3}{2}$
products of taking $2$ roots at a time, and $e_{_1}$ is the sum of all
$\binom{3}{1}$ `products' of taking $1$ root at a time.  Also define the power sums on $\mathcal{Z}$ per $p_i(\mathcal{Z})\equiv p_i :=\sum\limits_{z \in \mathcal{Z}} z^i$; $p_0:=|\mathcal{Z}|$.  With these definitions, equation (\ref{eqn:rsquaredbar}) above is equivalent to  $p_{_2}=\sum
r_{_k}^{\,2}=e_{_1}^2-2 e_{_2}$.  Such effort to
determine the sum of powers (or average of powers) is an established result,
the Girard (1629) and Waring (1762) formula
\cite{gould} which is very closely connected to Newton's
identities (ca.
1687) \cite{wikipolynomials}.  Both are based upon the
fact that for a (multi)set $\mathcal{Z}$ with $|\mathcal{Z}|=n$ it holds that
$\sum\limits_{\substack{i\,+\,j\,\,=\,n\\0\,\leq\,i,\,j\,\,\leq\,n}}\left(-1%
\right)^jp_{_j}(\mathcal{Z})e_{_i}(\mathcal{Z})=0$.  (Note that
$e_{_i}\left(\mathcal{R}\right)=\binom{n}{i}\,\overline{\,\,\,r}\!\!\!\!\!^{^i}\,\,\,\,$.) This relationship can either be solved for
the $e_{_i}$ or for the $p_{_j}$.  Solving for an $e_{_k}$ one obtains Newton's recursive
identities $e_{_k}=\frac{1}{k\,} \, \sum\limits_{\substack{i\,+\,j\,\,=\,k\\0\,<\,\,j\,\,\leq\,k}}\left(-1%
\right)^{j\,+\,1}p_{_j}e_{_i}$. Solving instead for $p_{_j}$, one obtains
the Girard-Waring formula.  To state the Girard-Waring formula compactly, 
Put $\boldsymbol{\mathfrak{n}}:=\{1, 2, \ldots , n\}$ 
and 
consider the $n$-tuples of natural numbers,
$\overset{\rightharpoonup}{\kappa} \equiv\left(k_{_{\,i}}\right)_{_{i\,
\in\,\boldsymbol{\mathfrak{n}}}}\in\mathbb{N}^n$ \  ($0\in\mathbb{N}$), and define
$\mathcal{K}_{_{\boldsymbol{\mathfrak{n}} j}}:=\left\{\overset{\rightharpoonup}{\kappa}\in
\mathbb{N}^n\,\Big|\,\,\left\|\overset{\rightharpoonup}{\kappa}\right
\|=j\right\}$ with
$\left\|\overset{\rightharpoonup}{\kappa}\right\|:=\sum\limits_{i\in \boldsymbol{\mathfrak{n}}}i\,k_{_{\,
i}}$, with $n$ being the number of data (or root)
values. Note that for $n\ge j, \ \mathcal{K}_{_{\boldsymbol{\mathfrak{n}} j}}$ is isomorphic to $\mathcal{K}_{_{
\large \boldsymbol{\mathfrak{j}} j}} =:\mathcal{K}_j$ which is the set of integer
partitions of $j$; the only difference (which is insignificant) is that elements of $\mathcal{K}_{_{\boldsymbol{\mathfrak{n}} j}}$ may have trailing zeroes in the partition.  E.g., $(3,0,0,1,0,0,0)\in \mathcal{K}_7$ denotes the partition $1+1+1+4$ of $7$. The Girard-Waring formula is that \ 
$p_{_j}=\,\,\sum\limits_{\overset{\rightharpoonup}{\kappa}\,\,\in\,
\mathcal{K}_{_j}}\gamma_{_{\,\overset{\rightharpoonup}{\kappa}}}\,\,\prod\limits_{i\,\in\,\boldsymbol{\mathfrak{n}}}\,\,e_{_i}\,^{k_{_i}}$, \  where \ $\gamma_{_{\,\overset{\rightharpoonup}{\kappa}}}:={j\,\left(-1\right)^j\,\frac{
\left(-1\right)^{\left|\overset{\rightharpoonup}{\kappa}\right|}\,}{
\left|\overset{\rightharpoonup}{\kappa}\right|}\binom{\left|
\overset{\rightharpoonup}{\kappa}\right|}{\overset{
\rightharpoonup}{\kappa}}}\,$, where
$\left|\overset{\rightharpoonup}{\kappa}\right|:=\sum\limits_{i\,\in\,
\boldsymbol{\mathfrak{n}}}k_{_i}$  and where
$\binom{\left|\overset{\rightharpoonup}{\kappa}\right|}{
\overset{\rightharpoonup}{\kappa}}=\frac{\,\left|\overset{
\rightharpoonup}{\kappa}\right|\,!\,}{\,\,\,\,\prod\limits_{i\,\in\,\boldsymbol{\mathfrak{n}}}k_{_i}!\,}\,
$ is the multinomial coefficient over the family of subindices $k_{_{\,i}}$
in $\overset{\rightharpoonup}{\kappa}$. (See Gould
\cite{gould}).  (The $\gamma_{_{\,\overset{\rightharpoonup}{\kappa}}}$ are given as sequence A210258
in Sloane's OEIS \cite{sloane:oeis}.) From our perspectives it
will become more natural to represent this relationship by replacing 
$p_{_j}$ and $e_{_i}$ by their normalized forms $\overline{p_{_j}}=p_{_j}/n$ and 
$\overline{\,e_{_i}}\,=\,\overline{\,\,\,r}\!\!\!\!\!^{^i}\,\,\,\,=\,e_{_i}
\big/\binom{n}{i}$, admittedly superficial changes to
Girard-Waring.  We thus obtain \quad
$\overline{p_{_j}}=\,\,\,\sum\limits_{\overset{\rightharpoonup}{
\kappa}\,\,\in\,\mathcal{K}_{_j}}c_{_{\,\overset{\rightharpoonup}{\kappa}}}\,\,\prod\limits_{i\,\in\,\boldsymbol{\mathfrak{n}}}\,\,\overline{\,\,\,r}\!\!\!\!\!^{^i}\,\,\,\,^{\,k_{_i}}$  where
\begin{equation}\label{cKappa}
 c_{_{\,\overset{\rightharpoonup}{\kappa}}}:=\frac{j\,\left(-1\right)^j\,
}{n}\frac{\left(-1\right)^{\left|\overset{
\rightharpoonup}{\kappa}\right|}\,}{\left|\overset{
\rightharpoonup}{\kappa}\right|}\binom{\left|\overset{
\rightharpoonup}{\kappa}\right|}{\overset{\rightharpoonup}{
\kappa}}\,\,\prod\limits_{i\,\in\,\boldsymbol{\mathfrak{n}}}\binom{n}{i}^{k_i}\text{ .}
\end{equation}
\noindent
We emphasize that for any partition of the power $j$, i.e. $\forall \left(k_i\right)_{i \in \boldsymbol{\mathfrak{n}}} \in \mathcal{K}_j$, the coefficient of the corresponding term $\prod\limits_{i \in \boldsymbol{\mathfrak{n}}}\,\overline{\,\,\,r}\!\!\!\!\!^{^i}\,\,\,\,^{\,k_{_i}}$ is given by this formula for $c_{_{\,\overset{\rightharpoonup}{\kappa}}}$.  By holding fixed all but one of the many parameters in equation (\ref{cKappa}), one obtains many sequences of coefficients.  E.g., the coefficients of $\overline{t}^{\,\ell}\ \overline{\overline{\overline{t}}}^{\,2}$ for a family of $n=3$ gives sequence $(1, 7, 36, 162, \ldots)$ (OEIS A080420 \ \cite{sloane:oeis}) with formula $\frac{j(j-5)}{18}\, 3^{j-5}$ (degree $j=\ell+6$).  E.g., the coefficients of $\overline{u}^{\,\ell}\, \overline{\overline{u}}$ for a family of $n=4$ gives sequence $(3, 18, 96, 480, \ldots)$ (not in OEIS) with formula $6\,j\, 4^{j-3}$ (degree $j=\ell+2$).  
Tabulated results are in Tables $\overline{\text{GW}}.\boldsymbol{\mathfrak{n}}$=2 through
$\overline{\text{GW}}.\boldsymbol{\mathfrak{n}}$=7 and in Tables $\overline{\text{GW}}.\deg$=2 through
$\overline{\text{GW}}.\deg$=7, at the end of the paper.  We particularly note that the coefficients in Table $\overline{\text{GW}}.\boldsymbol{\mathfrak{n}}$=2 are exactly the coefficients of the Tchebyshev polynomials of the first kind.  Thus our Tables $\overline{\text{GW}}.\boldsymbol{\mathfrak{n}}$ are generalizations of these Tchebyshev polynomials.

\section{Example}

An illustrative example is to consider the long-term goal of determining
relationships among \ $\varphi_{_{4,0,\rho}}\equiv \overline{{f_{_4}\left(\mathcal{R}^{\left(\,\rho\right)}\right)}}$ for
$\rho\in\left\{1,2,3\right\}$.  A better example would address
$\varphi_{_{6,0,\rho}}\equiv$
$\overline{{f_{_6}\left(\mathcal{R}^{\left(\,\rho\right)}\right)}}$ for
$\rho\in\left\{1,2,3,4,5\right\}$, since with $f'$ being a quintic we have no
hope of algebraically knowing those roots, but the larger example demands much
lengthier efforts.$\,\,$So begin with
$f_{_4}\left(x\right)=a\left(x^4-4\,\overline{r}\,x^3+6\,\overline{%
\overline{r}}\,x^2-4\,\overline{\overline{\overline{r}}}\,x+\overline{%
\overline{\overline{\overline{r}}}}\right)$. \ For $\rho=3$, we have$\,$the
cardinality $1$ family $\mathcal{R}'''=\left(\overline{r}\right)$, thus
quickly we obtain

$$\varphi_{_{4,0,3}}=f_{_4}\left(\overline{r}\right)=-a\left(3\,\overline{r}^{%
\,4}-6\,\overline{\overline{r}}\,\overline{r}^{\,2}+4\,\overline{\overline{%
\overline{r}}}\,\overline{r}\,-\overline{\overline{\overline{\overline{r}}}}%
\right).$$
\noindent
For $\rho=2$, let the family of the roots of the $2^{\text{nd}}$ derivative
be $\mathcal{R}''=\left(s_{_1},\,s_{_2}\right)$, and toward determining \ \
$\varphi_{_{4,0,2}}\equiv\overline{{f_{_4}\left(\mathcal{R}^{\left(2\right)}%
\right)}}\,$ let us consider
$f_{_4}\left(s_{_i}\right)= \linebreak a\left(s_{_i}^4-4\,\overline{r}\,s_{_i}^3+6\,%
\overline{\overline{r}}\,s_{_i}^2-4\,\overline{\overline{\overline{r}}}\,%
s_{_i}+\overline{\overline{\overline{\overline{r}}}}\right)$. \ From our mean
versions of  the Girard-Waring formula (presented as Table $\overline{\text{GW}}.\boldsymbol{\mathfrak{n}}$=2 at the end of
the paper) we have:  
\begin{center}
$\overline{s^2}\,=\,{2\,\,\overline{s}^{\,\,2}-1\,\overline{\overline{s}}}$, \ \ 
$\overline{s^3}\,=\,{4\,\,\overline{s}^{\,\,3}-3\,\overline{s}\,\,\overline{%
\overline{s}}}$, \  and \ 
$\overline{s^4}\,=\,8\,\overline{s}^{\,\,4}-8\,\overline{s}^{\,2}\,\overline{%
\overline{s}}\,+1\,\overline{\overline{s}}^{\,\,2}$. 
\end{center} 
\noindent
These give us

\begin{tabbing}
\hspace*{0.256in}\=\hspace{0.5in}\=\hspace{0.5in}\=\kill
\>{} $\overline{f_{_4}\left(s_{_i}\right)}=a\biggl(\left(8\,\overline{s}^{\,\,4}-8%
\,\overline{s}^{\,2}\,\overline{\overline{s}}\,+1\,\overline{\overline{s}}^{\,%
\,2}\right)-4\,\overline{r}\,\left({4\,\,\overline{s}^{\,\,3}-3\,\overline{s}%
\,\,\overline{\overline{s}}}\right)\,+$\\
\>{}\>{}\>{$\,6\,\overline{\overline{r}}\left(2\,\,\overline{s}^{\,\,%
2}-1\,\overline{\overline{s}}\right)-4\,\overline{\overline{\overline{r}}}\,%
\overline{s}+\overline{\overline{\overline{\overline{r}}}}\biggr)$. }
\end{tabbing}
\noindent 
We would be stuck here were it not for the fact that the $s_{_i}$ are roots
of a derivative of $f$, and thus we are blessed with the facts that
$\overline{s}=\overline{r}\,\,$and$\,\,{\overline{\overline{s}}\,=\overline{%
\overline{r}}}$. This is the key issue enabling general degree success. 
Thus we substantively extend Girard-Waring by considering 
$\overline{p_{_j}}\left(\mathcal{R}^{\left(m\right)}_{_f}\right)$ and
$\overline{\,e_{_i}}\,\left(\mathcal{R}_{_f}\right)$ where
$\mathcal{R}^{\left(m\right)}_{_f}$ is the roots family of the derivative
$f^{\left(m\right)}$.  Simplifying our current degree $4$ expression, we
now have \\[-0.75\baselineskip]
$$\varphi_{_{4,0,2}}= a\left(-8\,\,{\overline{r}}^{\,\,4}+16\,\,{\overline{r}}^{\,\,2}\,{
\overline{\overline{r}}}-4\,\,{\overline{r}}\,\,{\overline{\overline{
\overline{r}}}}-5\,\,{\overline{\overline{r}}}^{\,\,2}+1\,\,{\overline{
\overline{\overline{\overline{r}}}}}\,\right)\text{ .}$$

\noindent
For $\rho=1$, let the family of the roots of the $1^{\text{st}}$ derivative
be $\mathcal{R}'=\left(t_{_1},\,t_{_2},\,t_{_3}\right)$, \ and\linebreak toward
determining \ 
$\varphi_{_{4,0,1}}\equiv\overline{{f_{_4}\left(\mathcal{R}^{\left(1\right)}%
\right)}}\,$ \  let us consider \ 
$f_{_4}\left(t_{_i}\right)=\linebreak a\left(t_{_i}^4-4\,\overline{r}\,t_{_i}^3+6\,%
\overline{\overline{r}}\,t_{_i}^2-4\,\overline{\overline{\overline{r}}}\,%
t_{_i}+\overline{\overline{\overline{\overline{r}}}}\right)$. \ From our mean
versions of \ the Girard-Waring formula (Table $\overline{\text{GW}}.\boldsymbol{\mathfrak{n}}$=3) we have \ 
$\overline{t^2}\,=\,3\,\,\overline{t}^{\,\,2}-2\,\overline{\overline{t}}$, \ 
$\overline{t^3}\,=\,{9\,\,\overline{t}^{\,\,3}-9\,\overline{t}\,\,\overline{%
\overline{t}}\,+\overline{\overline{\overline{t}}}}$, and
$\overline{t^4}\,=\,27\,\overline{t}^{\,\,4}-36\,\overline{t}^{\,2}\,%
\overline{\overline{t}}\,+\,4\,\overline{t}\,\overline{\overline{%
\overline{t}}}\,\,+\,6\,\overline{\overline{t}}^{\,\,2}$.  These give us

\begin{tabbing}
\hspace*{0.256in}\=\hspace{0.5in}\=\kill
\>{}$\overline{f_{_4}\left(t_{_i}\right)}=a\biggl(\left(27\,\overline{t}^{\,\,%
4}-36\,\overline{t}^{\,2}\,\overline{\overline{t}}\,+\,4\,\overline{t}\,%
\overline{\overline{\overline{t}}}\,\,+\,6\,\overline{\overline{t}}^{\,\,2}%
\right)-\,$\\
\>{}\>{$4\,\overline{r}\,\left({9\,\,\overline{t}^{\,\,3}-9\,%
\overline{t}\,\,\overline{\overline{t}}\,+\overline{\overline{\overline{t}}}}%
\right)+\,6\,\overline{\overline{r}}\left(3\,\,\overline{t}^{\,\,2}-2\,%
\overline{\overline{t}}\right)-4\,\overline{\overline{\overline{r}}}\,%
\overline{t}+\overline{\overline{\overline{\overline{r}}}}\biggr)$. }
\end{tabbing}
\noindent
We would again be stuck here were it not for the fact that the $t_{_i}$ are
roots of a derivative of $f$, and thus we are blessed with the facts that
$\overline{t}=\overline{r}\,,\,\,\,{\overline{\overline{t}}\,=\overline{%
\overline{r}}},\,$and
\text{$\overline{\overline{\overline{t}}}\,=\overline{\overline{%
\overline{r}}}$}. Simplifying, we now have \\[-1.25\baselineskip]
$$\varphi_{_{4,0,1}}= a\left(-9\,\,{\overline{r}}^{\,\,4}+18\,\,{\overline{r}}^{\,\,2}\,{
\overline{\overline{r}}}-4\,\,{\overline{r}}\,\,{\overline{\overline{
\overline{r}}}}-6\,\,{\overline{\overline{r}}}^{\,\,2}+1\,\,{\overline{
\overline{\overline{\overline{r}}}}}\right)\text{ .}$$
\noindent
Henceforward we shall assume that the leading coefficient is $a=1$.  We have
summarized our
$\left(\varphi_{_{D,0,\rho}}\right)_{\rho\in\left\{1,2,\ldots,\,D-1\right\}}$
results for other degrees in Tables $\boldsymbol{\varphi.2}$ through $\boldsymbol{\varphi.7}$ at the end of the paper. \\

\noindent
\textbf{We summarize the above procedure.} \ For a given order $\rho$ of derivative,
consider the $\left(D-\rho\right)\,$roots of the derivative, and express the
normalized power sum means in terms of the derivative family's quasi-binomial
parameters, making use of our normalized variant of the Girard-Waring
formulas.  Then to evaluate the mean value of the original degree $D$
function over the derivative roots family, appropriately substitute these
mean power sum expressions in where the argument of the function occurs; this
takes advantage of the fact that averaging is a linear process, to wit, the
mean value of the polynomial function is the sum of the mean values of its
monomial components.  Next, taking advantage of the fact that the
quasi-binomial parameters of a family of derivative roots is the same (albeit
truncated) as the quasi-binomial parameters of the degree $D$ function, we
are able to simplify the expression of the $\varphi_{_{D,\delta,\rho}}$ to be
entirely in terms of the degree $D$ function's quasi-binomial (i.e.
normalized symmetric function) parameters. \\

With such procedure in hand, we are enabled to determine
general degree $D$ results \emph{ad libitum}. Tabulated results are in Tables
$\boldsymbol{\varphi.D}$ for $D=2$ through $D=7$, at the end of the paper. \\

Let us look a bit more closely at the results thus far obtained in the above
example.  We have:
\begin{tabbing}
\hspace{1.256in}\=\kill
\>{$\varphi_{_{4,0,1}}=$
$-9\,\,{\overline{r}}^{\,\,4}+18\,\,{\overline{r}}^{\,\,2}\,{\overline{%
\overline{r}}}-4\,\,{\overline{r}}\,\,{\overline{\overline{\overline{r}}}}-6\,%
\,{\overline{\overline{r}}}^{\,\,2}+1\,\,{\overline{\overline{\overline{%
\overline{r}}}}}$,}
\end{tabbing}

\begin{tabbing}
\hspace{1.256in}\=\kill
\>{$\varphi_{_{4,0,2}}=$
$-8\,\,{\overline{r}}^{\,\,4}+16\,\,{\overline{r}}^{\,\,2}\,{\overline{%
\overline{r}}}-4\,\,{\overline{r}}\,\,{\overline{\overline{\overline{r}}}}-5\,%
\,{\overline{\overline{r}}}^{\,\,2}+1\,\,{\overline{\overline{\overline{%
\overline{r}}}}}\,$,}
\end{tabbing}

\begin{tabbing}
\hspace{1.256in}\=\kill
\>{$\varphi_{_{4,0,3}}=$
$f_{_4}\left(\overline{r}\right)=-3\,\overline{r}^{\,4}+6\,\overline{%
\overline{r}}\,\overline{r}^{\,2}-4\,\overline{\overline{\overline{r}}}\,%
\overline{r}\,+1\,\,\overline{\overline{\overline{\overline{r}}}}$.}
\end{tabbing}

Observe that all of these have common terms
$-4\,\overline{\overline{\overline{r}}}\,\overline{r}\,+1\,\,\overline{
\overline{\overline{\overline{r}}}}$,  so there is more structure here than
what we have put our finger on. \ After sufficiently inspired inspection we
might realize the following unexpected relationship: \\[-0.5\baselineskip]
\begin{tabbing}
\hspace{0.256in}\=\hspace{0.5in}\=\hspace{0.5in}\=\hspace{0.5in}\=\kill
\>{}\>$5\,\,\varphi_{_{4,0,1}}\,-\,6\,\varphi_{_{4,0,2}}\,+\,1$
$\varphi_{_{4,0,3}}=0$ ,\\
i.e.\\[0.5\baselineskip]
\>{}\>{$5\,\overline{{f_{_4}\left(\mathcal{R}'\right)}}\,-6\,%
\overline{{f_{_4}\left(\mathcal{R}''\right)}}\,+1$
$\overline{{f_{_4}\left(\mathcal{R}'''\right)}}\,=0$ $\,\,$!}
\end{tabbing}
\noindent
\begin{remarks}
$\left(1\right)$ \ We desire a more systematic way, with less inspiration
required, to obtain such relationships.  In effect we are striving to solve
\text{$\alpha_{_1}\,\varphi_{_{4,0,1}}$}$+{\alpha_{_2}\,\varphi_{_{4,0,2}}}+%
\alpha_{_3}$ $\varphi_{_{4,0,3}}=0$, \ where the $\varphi_{_{4,0,\rho}}$ are
``vectors'' of linear combinations of the five ``basis elements'' \
$\left({\overline{r}}^{\,\,4},\,\,{\overline{r}}^{\,\,2}\,{\overline{%
\overline{r}}},\,\,{\overline{r}}\,\,{\overline{\overline{\overline{r}}}},\,\,%
{\overline{\overline{r}}}^{\,\,2},\,\,{\overline{\overline{\overline{%
\overline{r}}}}}\right)$. \ Thus the matter of finding all triplet
$\left(\alpha_{_1},\,\alpha_{_2},\,\alpha_{_3}\right)\,$solutions is a simple
linear algebra issue involving row reduction. \\[0.5\baselineskip]
\noindent
$\left(2\right)\,\,$At first glance, trying to solve
\text{$\alpha_{_1}\,\varphi_{_{4,0,1}}$}$+{\alpha_{_2}\,\varphi_{_{4,0,2}}}+%
\alpha_{_3}$ $\varphi_{_{4,0,3}}=0$ for the $\alpha_\rho$ seems like a
hopeless venture-- \ for this degree $4$ case, with $5$ basis elements we are
essentially trying to solve $5$ equations with only $3$ degrees of freedom
in our $\alpha$s.  But the fact that all three of the
$\varphi_{_{4,0,\rho}}$ have common terms
$-4\,\overline{\overline{\overline{r}}}\,\overline{r}\,+1\,\,\overline{%
\overline{\overline{\overline{r}}}}$ \ saves us:  if we can happen to
satisfy the other three basis elements with $\alpha$s such
that   $\sum\limits_{\rho\in\left\{1,2,3\right\}}\alpha_\rho=0$, \ then
the $\,\overline{\overline{\overline{r}}}\,\overline{r}\,$ and \
$\,\,\overline{\overline{\overline{\overline{r}}}}$ \ constraints will
automatically be satisfied.  Besides that, there is further structure within
our \ $\varphi_{_{4,0,\rho}}$ results-- observe that the coefficients of both
$\,\overline{r}^{\,4}$ and $\,\overline{\overline{r}}\,\overline{r}^{\,2}$
are in the proportion \
$\varphi_{_{4,0,1}}:\varphi_{_{4,0,2}}:\varphi_{_{4,0,3}}::9:8:3$. \ This
again increases our hope of finding at least one
$\left(\alpha_{_1},\,\alpha_{_2},\,\alpha_{_3}\right)\,$solution triplet. \\[0.5\baselineskip]
\noindent
$\left(3\right)$ \ Observe that: (a) the coefficients within the
$\varphi_{_{D,\delta,\rho}}$ are somewhat larger than the coefficients in the
eventual relationship
\text{$5\,\varphi_{_{4,0,1}}$}$-{6\,\varphi_{_{4,0,2}}}+{1\text{
}\varphi_{_{4,0,3}}}=0$; and (b) the number of basis elements inside each
$\varphi_{_{D,\delta,\rho}}$ is greater than the number of means
$\varphi_{_{D,\delta,\rho}}$ in our eventual relationship.  These are
typical situations.  Indeed, for relationships involving degrees up to $7$,
we sometimes see coefficients within some $\varphi_{_{D,\delta,\rho}}$ in the
tens of thousands, with as many as $15$ basis elements, yet the eventual
relationships remain small, with coefficients near $100$, and with $2$ to $6$
$\varphi_{_{D,\delta,\rho}}$ means involved.  The number of basis elements
for different degree polynomials is the number of integer partitions$\,$of
the degree (see OEIS  A$000041$ \cite{sloane:oeis}):
 
\begin{tabular}[t]{|@{}l|l|l|l|l|l|l|l|l|l|l|l|l|}
\hline
\hspace*{0.5cm}Degree&$2$&$3$&$4$&$5$&$6$&$7$&$8$&$9$&$10$&$11$&$12$&$13$\\
\hline
\,\# Basis Elements&$2$&$3$&$5$&$7$&$11$&$15$&$22$&$30$&$42$&$56$&$77$&$101$\\
\hline
\end{tabular}
\\[1\baselineskip]
This all strongly suggests that our proof procedure is unnecessarily
convoluted, and that some more natural and simple proof path is yet to be
found. \ Streamlined proofs still elude us. \\[0.5\baselineskip]
\noindent
$\left(4\right)$ \ Having noted in the chart above how quickly the number of
basis elements increases with increasing degree, we should be doubtful about
the prospects of finding
$\left(\alpha_{_1},\,\alpha_{_2},\ldots,\,\alpha_{_{D-1}}\right)$ solution
tuples for higher degree $D$ cases.  Our hope rests in there being ``special
circumstances'' similar to those noted in Remark $\left(2\right)\,$above. 
We shall find that such circumstances do hold. \\[0.5\baselineskip]
\end{remarks}
Before we give a resum\a'e of our complete results, we feel compelled to
present the striking degree $6$ result:  There is a unique solution
to $\sum\limits_{\rho\in\left\{1,\ldots,5\right\}}\alpha_{_\rho}\,\varphi_{_{6,0,%
\rho}}=0$, namely 
$$77\,\varphi_{_{6,0,1}} -120\,\varphi_{_{6,0,2}} +60\,
\varphi_{_{6,0,3}} -20\,\varphi_{_{6,0,4}}+3\, \varphi_{_{6,0,5}}=0\text{ .}$$
\\[-0.5\baselineskip]
\noindent
Perhaps equally surprising is that in degree $7$, even though constraints
from $15$ basis elements must be satisfied, we enjoy a $2$-dimensional
family of solution relationships.
\section{Graphical example in degree 4} See Figure \ref{fig:Quartic} at end of paper.

\section{Results}  The graphical example had us consider $\overline{{f'\left(\mathcal{R}\right)}}$ which is the average of the slopes at the roots, and this quantity is equal to the negation of twice the slope at the average of the roots.  Note that this latter quantity is invariant w.r.t. vertical translations, hence the average of the slopes at the roots does not depend upon the polynomial's constant term.  This is typical:

\begin{proposition}\label{prop:vertical_invariance}
If a horizontal line cuts a polynomial graph at as many points as the degree of the polynomial, the average of the slopes at the points of intersection is invariant w.r.t. modest vertical translations of the line.
\end{proposition}
\begin{proof}
(Sketch) Given a vertical translation $\Delta h$, one can compute changes in the various slopes, ignoring terms that are higher order in $\Delta h$.  Straightforward algebra shows that the sum of all such slope changes is $0$ for degrees $2$ through $7$, and we are confident that the same method works in all degrees.
\end{proof}
The interpretation of "modest" in the proposition is that the vertical translation should not cross an extremum.  In fact when one accounts for multiplicities and computes, if necessary, complex-valued derivatives for complex-value roots, it is seen that the modesty condition is not a requirement. \\

Our attempts at an inductive proof of the general case above led to the following confident conjecture (supported by strong numerical evidence).  The general proof of proposition \ref{prop:vertical_invariance} would follow as a corollary of the $k=2$ case of the following. Note that the conjecture gives statements in $(\text{reciprocal root units})^{k-1}$:

\begin{conj}
In any polynomial of degree $\ge 2$ with roots set $\mathcal{R}$ having no repeated roots, it holds that
\\[-1.5\baselineskip]
$$\forall k \ge 2, \qquad \sum\limits_{r \in \mathcal{R}}\frac{f^{(k)}(r)}{f'(r)} = 0 \text{ .}$$
\end{conj}

 \noindent 
Note that the above can be reinterpreted as:  $\forall \ell \ge 1$, the sum of relative rates is $ \sum\limits_{s \in \mathcal{S}_c}\frac{g^{(\ell)}(s)}{g(s)} = 0$ where $\mathcal{S}_c$ is the roots family of any antiderivative $\left(\int_0^x g \right)+c$ provided that $\mathcal{S}_c$ has no repeated roots and that $\deg(g)\ge 1$. \\

We now present some of our results. More completely see Tables $\boldsymbol{\alpha.4}$ through
$\boldsymbol{\alpha.6}$, as well as Tables $\boldsymbol{\varphi.2}$ through $\boldsymbol{\varphi.7}$, at the end of the
paper.

\begin{proposition}
The following is an exhaustive list of fundamental linear
relationships for degrees up to $8$, among means of a polynomial's values
when evaluated at roots of derivatives.

\begin{tabbing}
\hspace*{0.25in}\=\hspace{1.5in}\=\kill
\indent A. \ Degree$(f)=2$\>{}\>{N/A}
\end{tabbing}

B. \ Degree$(f)=3$
\\[-1.25\baselineskip]
\begin{tabbing}
\hspace*{0.25in}\=\hspace{0.4in}\=\hspace{2.5in}\=\kill
\>{$1 \,\, \varphi_{_{3,0,1}}-1\,\varphi_{_{3,0,2}}=0$}\>\>\textup{(a)}
\\[-0.25\baselineskip]
\end{tabbing}

C. \ Degree$(f)=4$
\\[-1.25\baselineskip]
\begin{tabbing}
\hspace*{0.25in}\=\hspace{0.4in}\=\hspace{2.5in}\=\kill
\>{ $5 \,\, \varphi_{_{4,0,1}}-6\,\varphi_{_{4,0,2}}+1\,\varphi_{_{4,0,3}}=0$}\>\>\textup{(b)}
\\[-0.25\baselineskip]
\end{tabbing}

D. \ Degree$(f)=5$
\\[-1.25\baselineskip]
\begin{tabbing}
\hspace*{0.25in}\=\hspace{0.4in}\=\hspace{2.5in}\=\kill
\>{ $1 \,\,\varphi_{_{5,0,1}}-\,\,3\,\varphi_{_{5,0,3}}+2\,\varphi_{_{5,0,4}}=0$} \>\>\textup{(c)}\\
\>{ $2 \,\,\varphi_{_{5,0,2}}\,-5\,\varphi_{_{5,0,3}}+\,\,3\,\varphi_{_{5,0,4}}=0$} \>\>\textup{(d)}\\
\>{ $3 \,\,\varphi_{_{5,0,1}}-\,\,4\,\varphi_{_{5,0,2}}\,\,+1\,\varphi_{_{5,0,3}}=0$}\\
\>{ $5 \,\, \varphi_{_{5,0,1}}\,-\,\,6\,\varphi_{_{5,0,2}}+1\,\varphi_{_{5,0,4}}=0$}
\\[0.5\baselineskip]

\>{ $1 \,\, \varphi_{_{5,0,1}}-2 \varphi_{_{5,0,2}}+2\,\varphi_{_{5,0,3}}-1\,\varphi_{_{5,0,4}}=0$}\\[0.75\baselineskip]
\>{(any 3 of the above are linearly dependent; any 2 are independent)}
\\[-0.25\baselineskip]
\end{tabbing}

E. \ Degree$(f)=6$
\\[0.5\baselineskip]\indent
$77\,\,\varphi_{_{6,0,1}}-120\,\,\varphi_{_{6,0,2}}+60\,\varphi_{_{6,0,3}}-20\,
\varphi_{_{6,0,4}}+3\,\varphi_{_{6,0,5}}=0$ \\

F. \ Degree$(f)=7$
\\[0.5\baselineskip]\indent
$85\,\, \varphi_{_{7,0,1}}-144\, \varphi_{_{7,0,2}}+90\,\varphi_{_{7,0,3}}-40\,\varphi_{_{7,0,4}}+9\, \varphi_{_{7,0,5}}=0$

$82\,\, \varphi_{_{7,0,1}}-135\, \varphi_{_{7,0,2}}+75\,\varphi_{_{7,0,3}}-25\,\varphi_{_{7,0,4}}+3\, \varphi_{_{7,0,6}}=0$

$77\,\, \varphi_{_{7,0,1}}-120 \,\varphi_{_{7,0,2}}+50\,\varphi_{_{7,0,3}}-15 \,
\varphi_{_{7,0,5}}+8\,\varphi_{_{7,0,6}}=0$

$67\,\,\varphi_{_{7,0,1}}-\,90\,\varphi_{_{7,0,2}}\,+\,50\,\varphi_{_{7,0,4}}\,
-45 \,\varphi_{_{7,0,5}}+18\,\varphi_{_{7,0,6}}=0$

$37\,\,\varphi_{_{7,0,1}}-150\,\varphi_{_{7,0,3}}+200\,\varphi_{_{7,0,4}}-135
\,\varphi_{_{7,0,5}}+48\,\varphi_{_{7,0,6}}=0$

$111\,\,\varphi_{_{7,0,2}}-335\,\varphi_{_{7,0,3}}+385\,\varphi_{_{7,0,4}}-246\,
\varphi_{_{7,0,5}}+85\,\varphi_{_{7,0,6}}=0$ 
\\[0.5\baselineskip]\indent
$1\,\,\varphi_{_{7,0,1}}-3\,\,\varphi_{_{7,0,2}}+5\,\,\varphi_{_{7,0,3}}-5\,\,
\varphi_{_{7,0,4}}+3\,\,\varphi_{_{7,0,5}}-1\,\,\varphi_{_{7,0,6}}=0$

\begin{tabbing}
\hspace{0.256in}\=\kill
\>(any 3 of the above are linearly dependent; any 2 are independent)
\end{tabbing}

F. \ Degree$(f)=8$
\\[0.5\baselineskip]\indent
$669\,\, \varphi_{_{8,0,1}}-1260\, \varphi_{_{8,0,2}}+1050\,\varphi_{_{8,0,3}}-700\,\varphi_{_{8,0,4}}+315\, \varphi_{_{8,0,5}} - \\[0.25\baselineskip]
\hspace*{0.85in} 84\, \varphi_{8,0,6} + 10\, \varphi_{8,0,7}=0$

\end{proposition}

\noindent
Notes: 
The alphabetical labels at the right margins in the preceding \& following
propositions indicate ``derivative inheritance'' relationships across
different degrees.  \\[0.5\baselineskip]
\noindent
In all the cases in the above list, the sum of all positive
coefficients equals the sum of all negative coefficients.  We have computed such relations up to degree 49, observing that (1) in all even degrees beyond 2 there is a single such fundamental linear relationship, and (2) in all odd degrees beyond 3 there is a 2-dimensional family of fundamental linear relationships. Moreover, the odd degrees enjoy the following clean relationship, which we regard as a main computational result with full confidence.
\begin{conj}
For all odd degree polynomials of degree $D$, it holds that  $$ \sum_{0<\rho<D} (-1)^\rho \binom{D}{\rho}\, \varphi_{D,0,\rho} =0 \text{ .}$$
\end{conj}

There are two ways to augment the above fundamental relationships with more
relationships:  (a) instead of considering values of the function $f$ we can
consider values of its derivatives $f^{\left(n\right)}$ or its (repeated)
antiderivatives, which we denote as $f^{\left(-\,n\right)}$; or (b) we can
expand the list of root families available by considering the root familes of
antiderivatives $f^{\left(-\,n\right)}$. When  using antiderivatives, it
turns out that there is no dependence on the constant of integration; this is a consequence of proposition \ref{prop:vertical_invariance}.

\begin{proposition}
Expressing polynomials as their quasi-binomial
representations, \ $\int f_{_D}=\frac{1}{\,D\,+\,1\,}f_{_{D\,+\,1}}$ where
the constant term of $f_{_{D\,+\,1}}\,$is the constant of integration $c$. \ In
other words,
$\overset{\rightharpoonup}{{\mathcal{R}^{\left(-1\right)}}}=
\overset{\rightharpoonup}{\mathcal{R}}\uplus\left(c\right)$. \
Restated yet again, if $\overset{\rightharpoonup}{\mathcal{S}}$ is the
quasi-binomial parameter vector of any antiderivative of $f_{_\mathcal{R}}$,
then
$\overset{\rightharpoonup}{\mathcal{R}}\leq\overset{
\rightharpoonup}{\mathcal{S}}$, \ i.e.
$\overset{\rightharpoonup}{\mathcal{R}}$ is a (strict) subvector of
$\overset{\rightharpoonup}{\mathcal{S}}$ via extension by the
constant(s) of integration. \ Further, if the domain variable of the
polynomial carries dimensional units, then the constant $c$ is
$(D+1)$-dimensional.

Dually, if $\overset{\rightharpoonup}{\mathcal{S}}$ is the
quasi-binomial parameter vector of any derivative of $f_{_\mathcal{R}}$, then
\ $\overset{\rightharpoonup}{\mathcal{S}}\leq
\overset{\rightharpoonup}{\mathcal{R}}$, \ i.e.
$\overset{\rightharpoonup}{\mathcal{S}}$ is a (strict) subvector of
$\overset{\rightharpoonup}{\mathcal{R}}$ via truncation. 
\end{proposition}
\noindent
Example:  Consider any polynomial $f_D,$
e.g.$\,{f_3\left(x\right)=x^3-3\,x^2\,\overline{r}+3\,x^1\,\overline{
\overline{r}}-\,\overline{\overline{\overline{r}}}}$ with roots family $\mathcal{R}$, and consider averaging
the values of the function over some set $\mathcal{Z}$, $x\in\mathcal{Z}$. \
We obtain
$\overline{f_3\left(\mathcal{Z}\right)}=\overline{x^3}-3 ~
\overline{x^2}$~$\overline{r}+3$~$\overline{x^1}$~$\overline{
\overline{r}}- ~ \overline{\overline{\overline{r}}}$.  The average values
of the powers $\overline{x^m}$ are obtained by our mean-value modification of
the Girard-Waring formulas.  The case of $\mathcal{Z}$ being the roots
family of a derivative of $f$ is straightforward, as noted in the dual statement in the
proposition, with $\overset{\rightharpoonup}{\mathcal{Z}}$ being a
subvector of $\overset{\rightharpoonup}{\mathcal{R}}$ via truncation.
The case of $\mathcal{Z}$ being the roots family of an antiderivative of $f$
requires greater attention:  In the example here, the highest power term
present in $f_3\left(x\right)$ is order $3$ (\& dimensionality $3$ if $x$ bears units), but the
constant of integration in $\int f$ is of dimensionality $4$;  hence the
constant of integration cannot enter into the computation of
$\overline{x^3}$, etc., with similar behavior in the general situation. 
Indeed, the values of $\overline{x^3}$, etc., only depend upon the entries in
$\overset{\rightharpoonup}{\mathcal{Z}}$  that were already present in
$\overset{\rightharpoonup}{\mathcal{R}}$.  In the detail of the
present example we have, following our notation for a family of $4$,
$\mathcal{Z}=\mathcal{F}_{_4}\equiv\mathcal{U}$, $\overline{u}=\overline{r}$, $\overline{\overline{u}}=\overline{\overline{r}}$ and $\overline{\overline{\overline{u}}}=\overline{\overline{\overline{r}}}$; thus 
%
\hspace*{.25 in}$\overline{f_3\left(\mathcal{Z}\right)}=1\,\overline{u^3}-3 ~
\overline{u^2} ~ \overline{r}+3$~$\overline{u^1}$~$\overline{
\overline{r}}-1 ~ \overline{\overline{\overline{r}}} $ \\[0.75\baselineskip]
\hspace*{0.5in}$= 1\left\{\,16\,\,\overline{u}^{\,\,3}-18\,\overline{u}\,\,\overline{
\overline{u}}\,+3\,\overline{\overline{\overline{u}}}\right\}-3$~$\left\{\,4\,
\,\overline{u}^{\,\,2}-3\,\overline{\overline{u}}\right\}\overline{r}+3 ~ 
\left\{\overline{u}^{\,1}\right\}  \overline{\overline{r}} ~ - ~ \overline{
\overline{\overline{r}}}$
\begin{tabbing}
\hspace{0.5in}\=\hspace{0.5in}\=\kill
\>{}$=\,\overline{r}^{\,\,3}\left\{16-12\right\}+\,\overline{r}\,\,%
\overline{\overline{r}}\,\left\{-18+9+3\right\}+$~$\overline{\overline{%
\overline{r}}}\,\left\{3-1\right\}$\\[0.5\baselineskip]
\>{}$=\,4\,\overline{r}^{\,\,3}-6\,\,\overline{r}\,\,\overline{%
\overline{r}}\,+2\,$~$\overline{\overline{\overline{r}}}\,$ \
$=\,2\,\mathbb{W}$.
\end{tabbing}
\noindent
Note, in particular, that the average value
$\varphi_{_{3,0,-1}}\equiv\overline{f_{_3}\left(\mathcal{R}^{\left(-1\right)}%
\right)}\,$does not at all involve the constant of integration that affects
the family $\mathcal{R}^{\left(-1\right)}$.  This behavior is typical.  Let
us emphasize: \ in the computation of mean function values
$\overline{f_{_D}\left(\mathcal{R}^{\left(-m\right)}\right)}$, individual
mean powers such as $\overline{u^k}$ only involve the quasi-binomial
	parameters of the original $\overset{\rightharpoonup}{\mathcal{R}}$ and
never involve the constant(s) of integration of
$\mathcal{R}^{\left(-m\right)}$.  Thus we have:

\begin{proposition}
The mean quantities \
$\varphi_{_{D,0,-m}}\equiv\overline{f_{_D}\left(\mathcal{R}^{\left(-m\right)}%
\right)}\,$ and \linebreak
$\varphi_{_{D,\delta,-m}}\equiv\overline{f^{\left(\delta\right)}_{_D}\left(%
\mathcal{R}^{\left(-m\right)}\right)}\,$ (with $\delta,m>0$) do not depend on
constant(s) of integration. 
\end{proposition}
 
\begin{proposition}
$\ \varphi_{_{D,1,0}}=D\cdot\varphi_{_{\left(D-1\right),0,-1}}$

\begin{tabbing}
\hspace{0.45in}\=\hspace{0.5in}\=\hspace{0.5in}\=\kill
\>{}\>{$\varphi_{_{D,\delta,0}}=\frac{\,D!\,}{\,\left(D\,-\,%
\delta\right)!}\cdot\varphi_{_{\left(D-\delta\right),0,-\delta}}$}
\end{tabbing}
\end{proposition}
\noindent
Of course, the above proposition can be read in reverse as
giving$\,\,\varphi_{_{D,0,-1}}=\frac{1}{\left(D\,+\,1\right)}\cdot\varphi_{_{%
\left(D\,+\,1\right),1,0}}\,\,$and
$\varphi_{_{D,0,-\rho}}=\frac{\,D!\,}{\left(D\,+\,\rho\right)!\,}\cdot%
\varphi_{_{\left(D\,+\,\rho\right),\rho,0}\text{ }}$(with $\rho>0$). \ These
facts are substantiated in our tables of $\varphi_{_{D,0,\rho}}$ expressions
by the fact that whenever $\rho<0$, the highest order quasi-binomial
parameter $\,\overline{\,\,\,r}\!\!\!\!\!^{^i}\,\,\,\,$ is always only order
$D$, never an order in excess of $D$. \\

\noindent
For the following five propositions, see also Tables $\boldsymbol{\alpha.4}$ through
$\boldsymbol{\alpha.6}$.

\begin{proposition}
The following is a partial list of linear relationships among
means of a polynomial's \ derivative values when evaluated at roots of
derivatives.

\begin{tabbing}
\hspace{0.25in}\=\hspace{0.2in}\=\hspace{1.5in}\=\kill
\>A. \>{Degree$(f)=2$} \>{N/A}\\
\>{ }
\\[-2.0\baselineskip]
\end{tabbing}

\begin{tabbing}
\hspace{0.25in}\=\hspace{0.4in}\=\hspace{2.5in}\=\kill
\>B. \ Degree$\left(f\right)=3$
\\[0.5\baselineskip] \indent
\,$1\,\, \varphi_{_{3,1,0}}+1 \,\,\varphi_{_{3,1,2}}=0$ \>\>\> \textup{(e)}
\\[0.25\baselineskip] 
\end{tabbing}
\indent

C. \ Degree$\left(f\right)=4$
\\[0.5\baselineskip] \indent
$\,1 \,\,\varphi_{_{4,1,0}}+ 2 \, \varphi_{_{4,1,2}}=0$ 

$\, 1 \,\,\varphi_{_{4,1,0}}+ 2\, \varphi_{_{4,1,3}}=0$

$\,1\,\, \varphi_{_{4,1,2}}-1\,\varphi_{_{4,1,3}}=0$ 
\\[0.25\baselineskip] \indent

D. \ Degree$(f)=5$
\\[0.5\baselineskip] \indent
\,  $5\,\, \varphi_{_{5,1,2}}-6 \,\varphi_{_{5,1,3}}+ 1 \, \varphi_{_{5,1,4}}=0$ \ 
\\[0.25\baselineskip] \indent
E. \ Degree$(f)=6$
\\[0.5\baselineskip] \indent
$\,\, 1\,\,\varphi_{_{6,1,2}} -3 \,\varphi_{_{6,1,4}}+ 2 \, \varphi_{_{6,1,5}}=0$ \ \ \ 

$\,\,2\,\,\varphi_{_{6,1,3}}-5 \,\varphi_{_{6,1,4}}+ 3 \, \varphi_{_{6,1,5}}=0$ 

$\,\,3 \,\, \varphi_{_{6,1,2}}-4 \, \varphi_{_{6,1,3}}+ 1 \, \varphi_{_{6,1,4}}=0$ 

$\,\,5 \,\, \varphi_{_{6,1,2}}-6 \, \varphi_{_{6,1,3}}+ 1 \, \varphi_{_{6,1,5}}=0$ 
\\[-0.25\baselineskip] \indent

\ $1\,\varphi_{_{6,1,2}}-2\, \varphi_{_{6,1,3}}+2 \, \varphi_{_{6,1,4}} -1 \, \varphi_{_{6,1,5}}=0$
\\[0.5\baselineskip] \indent
(any 3 of the above are linearly dependent; any 2 are independent)

\end{proposition}

\begin{proposition}
The following is a partial list of linear relationships among
means of a polynomial's $2^{\text{nd}}$ derivative values when evaluated at
roots of derivatives.

\begin{tabbing}
\hspace{0.4in}\=\hspace{1.5in}\=\kill
\indent
A. \>{Degree$(f)=2$}\>{N/A}\\
\>{$\,$}\\[-1\baselineskip]
\indent
B. \>{Degree$(f)=3$}  \>{N/A}
\end{tabbing}

C. \  Degree$\left(f\right)=4$
\\[0.25\baselineskip]\noindent
\hspace*{0.25in}$1\,\, \varphi_{_{4,2,0}}-2 \, \varphi_{_{4,2,1}}=0$ \\
\noindent
\hspace*{0.25in}$1 \,\, \varphi_{_{4,2,1}}+1 \, \varphi_{_{4,2,3}}=0$ \\
\noindent
\hspace*{0.25in}$1 \,\, \varphi_{_{4,2,0}}+2 \, \varphi_{_{4,2,3}}=0$ 
\\[0.5\baselineskip]\noindent
\hspace*{0.25in}$1\,\, \varphi_{_{4,2,0}}-1 \, \varphi_{_{4,2,1}}+1 \, \varphi_{_{4,2,3}}=0$ 
\\[1\baselineskip]\indent
D. \ Degree$\left(f\right)=5$
\\[-1.25\baselineskip]\indent
\begin{tabbing}
\hspace{0.25in}\=\kill
\>{$1 \,\, \varphi_{_{5,2,0}}+5 \, \varphi_{_{5,2,3}}=0$ \ }\\
\>{$1 \,\, \varphi_{_{5,2,1}}+2 \, \varphi_{_{5,2,3}}=0$}\\
\>{$1 \,\, \varphi_{_{5,2,3}}- 1\,  \varphi_{_{5,2,4}}=0$ } \\[.5\baselineskip]\indent
\>{$1\,\, \varphi_{_{5,2,0}}-2 \, \varphi_{_{5,2,1}}+2 \, \varphi_{_{5,2,3}}-1 \, \varphi_{_{5,2,4}}=0$}\\[-0.5\baselineskip]\indent
\end{tabbing}

E. \ Degree$(f)=6$
\\[0.25\baselineskip]\noindent
\hspace*{0.25in}$5\,\,\varphi_{_{6,2,3}}-6 \, \varphi_{_{6,2,4}}+ \, 1 \, \varphi_{_{6,2,5}}=0$ \ \ \ 
\\[1\baselineskip]\indent

\end{proposition}

\begin{proposition}
The following is a partial list of linear relationships
among means of a polynomial's $3^{\text{rd}}$ derivative values when
evaluated at roots of derivatives.

\begin{tabbing}
\hspace{0.4in}\=\hspace{1.5in}\=\kill
\indent
A. \>{Degree$(f)=2$}\>{N/A}\\
\indent
B. \>{Degree$(f)=3$} \>{N/A}\\
\indent
C. \>{Degree$(f)=4$} \>{N/A}
\end{tabbing}

D. \ Degree$\left(f\right)=5$
\\[-1.25\baselineskip] 
\begin{tabbing}
\hspace{0.25in}\=\kill
\>$1 \,\, \varphi_{_{5,3,0}}-3\, \varphi_{_{5,3,2}}=0$ 
\\[-0.5\baselineskip] 
\end{tabbing}

\indent
E. \ Degree$\left(f\right)=6$
\\[0.5\baselineskip]\indent 
$\,1 \,\, \varphi_{_{6,3,0}}+\,\,9\, \varphi_{_{6,3,4}}=0$

$\,1\,\,\varphi_{_{6,3,1}}+\,\,5\, \varphi_{_{6,3,3}}=0$

$\,1\,\,\varphi_{_{6,3,2}}+2\, \varphi_{_{6,3,4}}=0$

$\,1\,\, \varphi_{_{6,3,4}}\,-1\, \varphi_{_{6,3,5}}=0$   

\end{proposition}

\begin{proposition}
The following is a partial list of linear relationships
among means of a polynomial's $4^{\text{th}}$ derivative values when
evaluated at roots of derivatives.

\begin{tabbing}
\hspace{0.4in}\=\hspace{1.5in}\=\kill
\indent
A. \>{Degree$(f)=2$} \>{N/A}\\ \indent
B. \>{Degree$(f)=3$} \>{N/A}\\ \indent
C. \>{Degree$(f)=4$} \>{N/A}\\ \indent 
D. \>{Degree$(f)=5$} \>{N/A}\\[-1\baselineskip]
\end{tabbing}

E. \ Degree$(f)=6$

$\,3 \,\, \varphi_{_{6,4,0}}\,-\,4 \,\varphi_{_{6,4,1}}\,=0$

$\,1\,\, \varphi_{_{6,4,1}}\,-\,3 \, \varphi_{_{6,4,3}}\,=0$

$\,1\,\, \varphi_{_{6,4,2}}\,-2 \, \varphi_{_{6,4,3}}\,=0$

$\,1\,\, \varphi_{_{6,4,3}}\,+1 \, \varphi_{_{6,4,5}}\,=0$  

\end{proposition}

\begin{proposition}
The following is a partial list of linear relationships
among means of a cubic polynomial's function, derivatives, and/or
antiderivatives values when evaluated at roots of derivatives and 
antiderivatives. \\

\ Degree$(f)=3$: 

$2\,\,\varphi_{_{3,-2,0}}\,$ $-5\,\,\varphi_{_{3,-2,1}}\, +3\,\,\varphi_{_{3,-2,2}}\,=0$

$1\,\,\varphi_{_{3,-2,-1}}\,-3\,\,\varphi_{_{3,-2,1}}\,+2\,\,\varphi_{_{3,-2,2}}\,=0$

$5\,\,\varphi_{_{3,-1,0}}\,-6\,\,\varphi_{_{3,-1,1}}\,+1\,\varphi_{_{3,-1,2}}\,=0$

$1 \,\, \varphi_{_{3,0,-1}}\,+2\,\, \varphi_{_{3,0,1}}\,=0$

$1 \,\, \varphi_{_{3,0,-2}}\,+5\,\, \varphi_{_{3,0,1}}\,=0$

$1 \,\, \,\varphi_{_{3,0,-3}}\,+\,9\,\, \varphi_{_{3,0,1}}\,=0$

$1 \,\, \varphi_{_{3,1,-1}}\,\,-2\,\, \varphi_{_{3,1,0}}\,=0$

$1 \,\, \varphi_{_{3,1,-2}}\,\,-\,3\,\, \varphi_{_{3,1,0}}\,=0$

$3 \,\, \varphi_{_{3,1,-3}}\,-\,4\,\, \varphi_{_{3,1,-2}}\,=0$
\end{proposition}
\noindent
The last six of the above deserve particular comment.  In the middle three
above, in the first terms (with coefficient $1$) the cubic function $f$ is
being averaged over the root families $\mathcal{R}^{\left(-1\right)},$
$\mathcal{R}^{\left(-2\right)},$ $\mathcal{R}^{\left(-3\right)}$, which are
the root families of $\int f,\iint f,\,$and$\iiint f\,$.  As noted in an
earlier proposition, these results are independent of choice of constants of
integration. \\

The following essentially restate above results with a different perspective.
 See Tables $\boldsymbol{\alpha.4}$ through $\boldsymbol{\alpha.6}$, \\[\baselineskip]
\noindent
Note:  Regarding the alphabetical labels to the right of some relationships, we have used double-letter hybrid
labels to indicate that the relationship is linearly dependent upon the
denoted pair of preceding relationships, in cases of worthy of attention since the relationship coefficients vector is highly symmetric,

\begin{proposition} In any quartic polynomial the following relations hold:

\begin{tabbing}
\hspace{0.25in}\=\hspace{0.4in}\=\hspace{2.5in}\=\kill
\>{$(0)$} \>{$5 \,\, \varphi_{_{4,0,1}}\,-6\,\varphi_{_{4,0,2}}+ 1\,\varphi_{_{4,0,3}}=0 $}\>{\textup{(b)}} \\

\>{$(1)$} \>{$1\,\varphi_{_{4,1,0}}+ 2 \, \varphi_{_{4,1,2}}=0$ }\>{\textup{(f)}}\\
\>{} \>{$1\,\, \varphi_{_{4,1,2}}-$ $1\,\,\varphi_{_{4,1,3}}=0$
\ }\>{\textup{(a')}$=D$\textup{(a)}} \\

\>{$(2)$} \>{$1\,\varphi_{_{4,2,0}}-2$ $\,\varphi_{_{4,2,1}}=0$
\ }\>{\textup{(g)}}\\
\>{} \>{$\,1\,\varphi_{_{4,2,1}}+$ $1\,\varphi_{_{4,2,3}}=0$
.}\>{\textup{(e')}$=D$\textup{(e)}}\\
\end{tabbing}
\end{proposition}

\begin{proposition}In any quintic polynomial the following relations hold:

\begin{tabbing}
\hspace{0.25in}\=\hspace{0.4in}\=\hspace{3.0in}\=\kill
\>{$(0)$}  \>{$1\,\, \varphi_{_{5,0,1}}-3 \, \varphi_{_{5,0,3}}+2 \, \varphi_{_{5,0,4}}=0$}  \>{\textup{(c)}}\\
\>{}  \>{$2 \,\, \varphi_{_{5,0,2}}-5 \,\varphi_{_{5,0,2}}+3 \, \varphi_{_{5,0,4}}=0$}  \>{\textup{(d)}}\\[0.5\baselineskip]
\>{} \>{$1\,\, \varphi_{_{5,0,1}}-\,2 \, \varphi_{_{5,0,2}}+2 \,\varphi_{_{5,0,3}}-1\,\varphi_{_{5,0,4}}=0$}\>{\textup{(cd)}} \\[1\baselineskip]
 
\>{$(1)$} \>{$5 \,\, \varphi_{_{5,1,2}}-6 \, \varphi_{_{5,1,3}}+ 1\,\varphi_{_{5,1,4}}=0$} \>{\textup{(b')}$=D$\textup{(b)}} \\[1\baselineskip]
 
\>{$(2)$} \>{$1\,\,\varphi_{_{5,2,0}}+ 5 \, \varphi_{_{5,2,3}}=0$ \ \ } \>{\textup{(h)}}\\
\>{}  \>{$1\,\, \varphi_{_{5,2,1}}+ 2 \,\varphi_{_{5,2,3}}=0$} \>{\textup{(f')}$=D$\textup{(f)}}\\
\>{}  \>{$1\,\, \varphi_{_{5,2,3}}- 1\,\varphi_{_{5,2,4}}=0$}  \>{\textup{(a'')}$=D$\textup{(a')}$=D^2$\textup{(a)}} \\[1\baselineskip]
\>{$(3)$}  \>{$1\,\, \varphi_{_{5,3,0}}-\,3 \, \varphi_{_{5,3,2}}=0$}  \>{ \textup{(i)}} \\
\>{}  \>{$1 \,\, \varphi_{_{5,3,1}}-2 \, \varphi_{_{5,3,2}}=0$}  \>{\textup{(g')}$=D$\textup{(g)}}\\
\>{}  \>{$1 \,\, \,\varphi_{_{5,3,2}}+1\, \varphi_{_{5,3,4}}=0$.} \>{\textup{(e'')}$=D$\textup{(e')}$=D^2$\textup{(e)}}\\

\end{tabbing}
\end{proposition}

\noindent
Also note that item $(0\text{f})$ states that for any quintic $f$,
the average of the function values at the two roots of $f'''$ equals the
$2:1$ weighted average of $(i)$ the value of the function at the
sole root of $f^{(4)}$ and $(ii)$ the average of the
value of the function at the four roots of $f'$.  Item
$(0\text{g})$ states that for any quintic $f$, the average of the
function values at the two roots of $f'''$ also equals the $3:2$ weighted
average of $(i)$ the value of the function at the sole root of
$f^{(4)}$ and $(ii)$ the average of the value of the
function at the three roots of $f''$.  We leave the remaining
interpretations to the reader.
\begin{proposition}
In any sextic polynomial the following relations hold:

\begin{tabbing}
\hspace{0.25in}\=\hspace{0.4in}\=\hspace{1.5in}\=\hspace{1in}\=\hspace{1.5in}\=%
\hspace{0.5in}\=\hspace{0.5in}\=\kill
\>{$(0)$}  \>{$77\,\, \varphi_{_{6,0,1}}-120\,\varphi_{_{6,0,2}}+60 \, \varphi_{_{6,0,3}}-  20 \, \varphi_{_{6,0,4}}+3 \, \varphi_{_{6,0,5}}=0$}  
\>{} \>\>{\textup{(j)}}\\[1\baselineskip]
 
\>{$(1)$} \>{$1\,\, \varphi_{_{6,1,2}}-3 \, \varphi_{_{6,1,4}}+2 \, \varphi_{_{6,1,5}}=0$}    \>\>{\textup{(c')}$=D$\textup{(c)}}\\

\>{}  \>{$2 \,\, \varphi_{_{6,1,3}}-5 \, \varphi_{_{6,1,4}}+3 \, \varphi_{_{6,1,5}}=0$ }  \>\>{\textup{(d')}$=D$\textup{(d)}}\\[1\baselineskip]
 
\>{$(2)$} \>{$5\,\, \varphi_{_{6,2,3}}-6\,\varphi_{_{6,2,4}}+1\,\varphi_{_{6,2,5}}=0$}   \>\>{\textup{(b'')}$=D$\textup{(b')}$=D^2$\textup{(b)}}\\[1\baselineskip]
 
\>{$(3)$}  \>{$1\,\, \varphi_{_{6,3,0}}+9\varphi_{_{6,3,4}}=0$} \>\>{\textup{(k)}} \\

\>{}  \>{$1\,\,\varphi_{_{6,3,1}}+5 \, \varphi_{_{6,3,4}}=0$} \>\>{\textup{(h')}$=D$\textup{(h)}} \\

\>{} \>{$1\,\, \varphi_{_{6,3,2}}+2 \, \varphi_{_{6,3,4}}=0$} \>\>{\textup{(f'')}$=D$\textup{(f')}$=D^2$\textup{(f)}} \\

\>{} \>{$1\,\, \varphi_{_{6,3,4}}- 1\,\varphi_{_{6,3,5}}=0$} \>\>{\textup{(a''')}$=D$\textup{(a'')}$=D^2$\textup{(a')}$=D^3$\textup{(a)}}\\[1\baselineskip]
 
\>{$(4)$}  \>{$3\,\, \varphi_{_{6,4,0}}-4 \, \varphi_{_{6,4,1}}=0$ } \>\>{\textup{(l)}} \\ 

\>{}  \>{$1\,\, \varphi_{_{6,4,1}}-\,3\,\varphi_{_{6,4,3}}=0$ } \>\>{\textup{(i')}$=D$\textup{(i)}} \\

\>{}  \>{$1\,\, \varphi_{_{6,4,2}}-2 \, \varphi_{_{6,4,3}}=0$} \>\>{\textup{(g'')}$=D$\textup{(g')}$=D^2$\textup{(g)}} \\

\>{} \>{$1\,\,\varphi_{_{6,4,3}}+ 1\,\varphi_{_{6,4,5}}=0$} \>\>{\textup{(e''')}$=D$\textup{(e'')}$=D^2$\textup{(e')}$=D^3$\textup{(e)}}\\
\end{tabbing}
\end{proposition}

\begin{proposition}
In any septic polynomial the following relations hold:

\begin{tabbing}
\hspace{0.256in}\=\hspace{0.5in}\=\hspace{0.5in}\=\hspace{0.5in}\=%
\hspace{0.5in}\=\hspace{0.5in}\=\hspace{0.5in}\=\hspace{0.5in}\=%
\hspace{0.5in}\=\hspace{0.5in}\=\hspace{0.5in}\=\hspace{0.5in}\=\kill
\hspace{0.25in}\=\hspace{0.4in}\=\hspace{2.5in}\=\hspace{1.5in}\=%
\hspace{0.5in}\=\hspace{0.5in}\=\kill
\>{ $(0)$}  \>{$85\,\, \varphi_{_{7,0,1}}-144\, \varphi_{_{7,0,2}}+90 \, \varphi_{_{7,0,3}}-     40 \, \varphi_{_{7,0,4}}+9 \, \varphi_{_{7,0,5}}=0$} 
\>{}  \>{\textup{(m)}} \\[0\baselineskip]

\>{}  \>{$82\,\, \varphi_{_{7,0,1}}-135\,\varphi_{_{7,0,2}}+75 \, \varphi_{_{7,0,3}}-25 \, \varphi_{_{7,0,4}}+3 \, \varphi_{_{7,0,6}}=0$} \>{}\>\textup{(n)}  \\[0.5\baselineskip]
\>{}  \>{$1\,\, \varphi_{_{7,0,1}}-3\, \varphi_{_{7,0,2}}+5 \, \varphi_{_{7,0,3}}- 5 \, \varphi_{_{7,0,4}}+3 \, \varphi_{_{7,0,5}}-1 \, \varphi_{_{7,0,6}}=0$} 
\>{}\>{\textup{(mn)}}
\end{tabbing}
\end{proposition}

We find these relationships
distinctively impressive.  There are many more identities, not recorded above, that are present when we
open up the limitless world of roots $\,\,\mathcal{R}^{\left(-n\right)}$ of
$f^{\left(-n\right)}:=\underset{n \text{copies}}{{\underbrace{{\iint\cdots%
\int}}}}f$, i.e. when we consider the limitless world of mean values 
$\varphi_{_{D,\delta,\rho}}$ where $\rho<0$.   See Tables $\boldsymbol{\alpha.4}$ through
$\boldsymbol{\alpha.6}$.   \\[-1\baselineskip]
\subsection{Ancillary Computational Results}
Consider the following observations regarding the coefficients of $\overline{\,\,\,\,\,\,r}\!\!\!\!\!\!\!\!\!^{^\text{max}}\,\,\,\,\equiv \overline{\,\,\,r}\!\!\!\!\!\!^{^{D}}\,\,\,\,$ in $\varphi_{_{D,0,\rho}}$.  Let $h_D(n)$ be the polynomial in $n$ predicting the coefficient of $\overline{\,\,\,r}\!\!\!\!\!\!^{^{D}}\,\,\,\,$ for the $\boldsymbol{\varphi\text{.D}}$ family  $\left(\varphi_{_{D,0,\rho}}\right)_{\rho< D-1}$.  This polynomial $h_D(n)$ is always of the form $h_D(n)=\frac{(-1)^{D}D}{D!}\rho\ n^\chi g_D(n)$ where $g_D(n)$ is a monic irreducible (over $\mathbb{Z}$) polynomial of degree $M:=D-(2+\chi)$ with integer coefficients, where $\chi=1$ when $D$ is odd and $\chi=0$ when $D$ is even.  
Let $g_D(n)=\sum\limits_{k=2\ldots D}t_k(D)\, n^{D-(k+\chi)}$.  Note that $t_2(D)\equiv 1$ (monic condition); also for odd $k$, it holds that $t_k(D)=0$ \& $t_{k+1}(D)=0$ if $D\le k$.  The $t_k(D)$ are of the form $\frac{1}{Q_k} u_k(D)$ where $Q_k$ is the least common denominator of the terms in $t_k(D)$, such that $u_k(D)$ is a (generally not monic) polynomial in $D$ with integer coefficients, with leading coefficient denoted as $\mathfrak{N}_k$.
Curiously, denominators $Q_k$ are the OEIS \cite{sloane:oeis} sequence A053657, described in OEIS as  ``Denominators of integer-valued polynomials on prime numbers (with degree n): 1/a(n) is a generator of the ideal formed by the leading coefficients of integer-valued polynomials on prime numbers with degree less than or equal to n'', or equivalently ``Also the least common multiple of the orders of all finite subgroups of $GL_n(\mathbb{Q})$ [Minkowski]".  Strikingly, the leading coefficients $\mathfrak{N}_k$ of the polynomials $u_k(D)$  are coefficients of N{\o}rlund's polynomials, see OEIS A260326 (which had listed only 7 values previous to our work), described there as ``Common denominator of coefficients in N{\o}rlund's polynomial D\_{2n}(x).''  [The OEIS citation refers to N{\o}rlund's 1924 book \cite{norlund} (which discusses the higher order Bernoulli \& Euler polynomials), Table 6 (p. 460), which only lists 7 values, for \emph{even} indices 0 through 12, but Table 5 (p. 459) includes both even \& odd entries as the leading coefficients of the primary components of the numerator polynomials. Our work has produced 198 terms with values eventually exceeding $5.8 \times 10^{40}$.]\\

This paper is a spin-off from, and is material to, our recent
\emph{Insights Via Representational Naturality: New Surprises Intertwining Statistics, Cardano's Cubic Formula,
Triangle Geometry, Polynomial Graphs} \cite{the authors}.  
\\[\baselineskip] \indent
Here is a summary of our tabulated results. \\[0.5\baselineskip]
\noindent
\textbf{Tables $\boldsymbol{\varphi}$.D}: \ Function value means $\varphi_{_{D,\delta,\rho}}$ in
terms of the  polynomial representation parameters $\overline{\,\,\,r}\!\!\!\!\!^{^i}\,\,\,\,$.  These tables
extend the Girard-Waring formula. \\[\baselineskip]
\noindent
\textbf{Tables $\boldsymbol{\alpha}$.D}: \ Linear Relationship Coefficients
$\alpha_{_{D,\delta,\rho}}$ such that ${\sum\limits_\rho\text{
}\alpha_{_{D,\delta,\rho}}\, \varphi_{_{D,\delta,\rho}}\,=0\ }$.  Some
$(D,\delta,\rho)$ triplets admit more than one
$\left(\alpha_{_{D,\delta,\rho}}\right)$ family. \\[1\baselineskip]
\noindent
\textbf{Tables $\overline{\text{GW}}.\boldsymbol{\mathfrak{n}}${=}$\boldsymbol{2}$ through $\overline{\text{GW}}.\boldsymbol{\mathfrak{n}}$=$\boldsymbol{7}$}: 
Normalized Girard-Waring coefficients \ $c_{_{\,\overset{\rightharpoonup}{\kappa}}}:=\frac{j\,\left(-1\right)^j\,}{n} \frac{\left(-1\right)^{\left|\overset{\rightharpoonup}{\kappa}\right|}\,}{\left|\overset{
\rightharpoonup}{\kappa}\right|}\binom{\left|\overset{
\rightharpoonup}{\kappa}\right|}{\overset{\rightharpoonup}{
\kappa}}\,\,\prod\limits_{i\,\in\,\boldsymbol{\mathfrak{n}}}\binom{n}{i}^{k_i}$ \ where $j=$ degree
and where $\,\overset{\rightharpoonup}{\kappa}\equiv\langle
k_1,k_2,\ldots,k_{_n}\rangle$ is an integer partition vector of $j$, i.e.
$\sum_i k_{_i}\cdot i=j$. \\[1\baselineskip]
\noindent
\textbf{Tables $\overline{\text{GW}}$.deg=$\boldsymbol{2}$ through $\overline{\text{GW}}$.deg=$\boldsymbol{7}$}:
\ Normalized Girard-Waring coefficients \ $c_{_{\,\overset{\rightharpoonup}{\kappa}}}:=\frac{j\,\left(-1\right)^j\,%
}{n}\frac{\left(-1\right)^{\left|\overset{
\rightharpoonup}{\kappa}\right|}\,}{\left|\overset{
\rightharpoonup}{\kappa}\right|}\binom{\left|\overset{
\rightharpoonup}{\kappa}\right|}{\overset{\rightharpoonup}{
\kappa}}\,\,\prod\limits_{i\,\in\,\boldsymbol{\mathfrak{n}}}\binom{n}{i}^{k_i}$ \ 
as above.
\newpage
\begin{landscape}
\begin{figure}[t]
\includegraphics[width=1.7\textwidth, clip=true,trim= 0cm 4cm 0cm 4cm]
{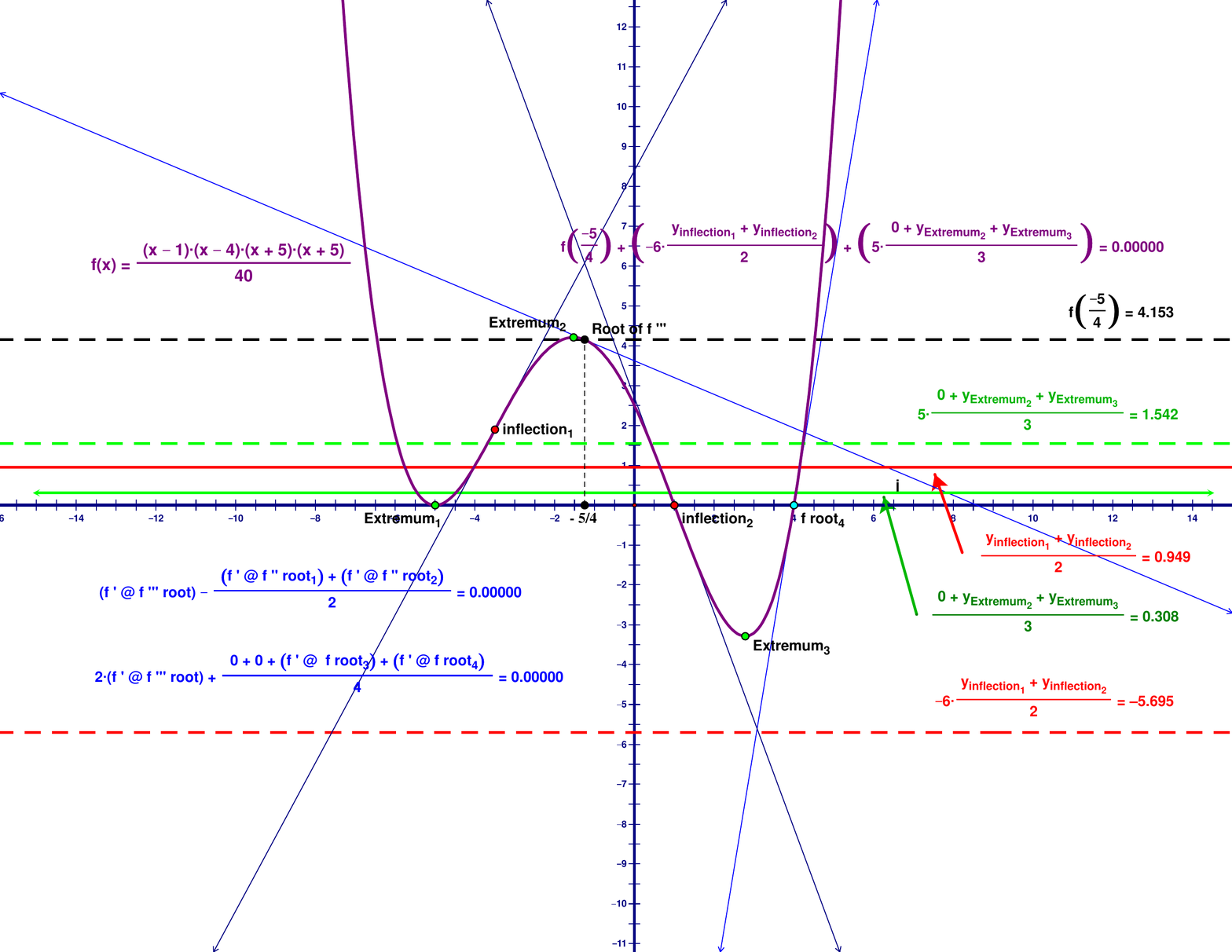}
\caption{%
Mean Relationships in all quartics.  This example manifests:  \\ 
\large
\hspace*{0.6in} (a) { $5
\overline{{f\left(\mathcal{R}'\right)}}-6\,\overline{{f\left(\mathcal{R}''%
\right)}}+1\overline{{f\left(\mathcal{R}'''\right)}}=0$} \qquad [note:  $\mathcal{R}'''=\{-5/4\}$] \\ 
\hspace*{0.6in} (b) { $1 \overline{{f'\left(\mathcal{R}'''\right)}}-1\overline{{f'\left(\mathcal{R}''\right)}}=0$} \\ \hspace*{0.6in} (c) { $2
\overline{{f'\left(\mathcal{R}'''\right)}}+1\overline{{f'\left(\mathcal{R}\right)}}=0$}
 }
\label{fig:Quartic}
\end{figure}
\end{landscape}

\newpage
\textbf{Supportive Results: \ Function Value Means \ $\varphi_{_{D\,\delta\,\rho}}$
in Terms of Polynomial Representation Parameters \
$\,\overline{\,\,\,r}\!\!\!\!\!^{^i}\,\,\,\,$ .}
\\[1\baselineskip]
\textbf{Table} $\boldsymbol{\varphi.2}$ \\[-0.5\baselineskip]
\begin{tabular}[t]{|l|l|l|l|}
\hline
$\genfrac{}{}{0pt}{}{\text{\# roots}}{n}$&$\rho$&$\varphi_{_{D\,\delta\,\rho}}=%
\varphi_{_{2\,0\,\rho}}$&$\sum_{\text{coef}>0}$\\
\hline
$1$&$1$&$-\left(\overline{r}^{\,\,2}-\overline{\overline{r}}\right)$&$1$\\
\hline
$2$&$0$&$0$&$0$\\
\hline
$3$&$-1$&$\,\,\,\,\,\,\,\,\overline{r}^{\,\,2}-\overline{\overline{r}}$&$1$\\
\hline
$4$&$-2$&$2\left(\,\overline{r}^{\,\,2}-\,\overline{\overline{r}}\right)$&$2$%
\\
\hline
$5$&$-3$&$3\,\left(\,\overline{r}^{\,\,2}-\,\overline{\overline{r}}%
\right)$&$3$\\
\hline
$6$&$-4$&$4\left(\,\overline{r}^{\,\,2}-\,\overline{\overline{r}}\right)$&$4$%
\\
\hline
$7$&$-5$&$5\,\left(\,\overline{r}^{\,\,2}-\,\overline{\overline{r}}%
\right)$&$5$\\
\hline
$8$&$-6$&$6\left(\,\overline{r}^{\,\,2}-\,\overline{\overline{r}}\right)$&$6$%
\\
\hline
$9$&$-7$&$7\,\left(\,\overline{r}^{\,\,2}-\,\overline{\overline{r}}%
\right)$&$7$\\
\hline
$10$&$-8$&$8\,\left(\,\overline{r}^{\,\,2}-\,\overline{\overline{r}}%
\right)$&$8$\\
\hline
\end{tabular}%
\hspace{0.5in}
\\[2\baselineskip]
\textbf{Table}  $\boldsymbol{\varphi.3}$\\[-0.5\baselineskip]
\begin{tabular}[t]{|l|l|l|l|}
\hline
$\genfrac{}{}{0pt}{}{\text{\# roots}}{n}$&$\rho$&$\varphi_{_{D\,\delta\,\rho}}=%
\varphi_{_{3\,0\,\rho}}$&$\sum_{\text{coef}>0}$\\
\hline
$1$&$2$&$-\left(2\,{\overline{r}}^{\,\,3}-3\,\overline{r}^{\,\,}\,\overline{%
\overline{r}}+\overline{\overline{\overline{r}}}\right)$&$3$\\
\hline
$2$&$1$&$-\left(2\,{\overline{r}}^{\,\,3}-3\,\overline{r}^{\,\,}\,\overline{%
\overline{r}}+\overline{\overline{\overline{r}}}\right)$&$3$\\
\hline
$3$&$0$&$0$&$0$\\
\hline
$4$&$-1$&$2\left(2\,{\overline{r}}^{\,\,3}-3\,\overline{r}^{\,\,}\,\overline{%
\overline{r}}+\overline{\overline{\overline{r}}}\right)$&$6$\\
\hline
$5$&$-2$&$5\left(2\,{\overline{r}}^{\,\,3}-3\,\overline{r}^{\,\,}\,\overline{%
\overline{r}}+\overline{\overline{\overline{r}}}\right)$&$15$\\
\hline
$6$&$-3$&$9\left(2\,{\overline{r}}^{\,\,3}-3\,\overline{r}^{\,\,}\,\overline{%
\overline{r}}+\overline{\overline{\overline{r}}}\right)$&$27$\\
\hline
$7$&$-4$&$14\left(2\,{\overline{r}}^{\,\,3}-3\,\overline{r}^{\,\,}\,%
\overline{\overline{r}}+\overline{\overline{\overline{r}}}\right)$&$42$\\
\hline
$8$&$-5$&$20\left(2\,{\overline{r}}^{\,\,3}-3\,\overline{r}^{\,\,}\,%
\overline{\overline{r}}+\overline{\overline{\overline{r}}}\right)$&$60$\\
\hline
$9$&$-6$&$27\left(2\,{\overline{r}}^{\,\,3}-3\,\overline{r}^{\,\,}\,%
\overline{\overline{r}}+\overline{\overline{\overline{r}}}\right)$&$81$\\
\hline
$10$&$-7$&$35\left(2\,{\overline{r}}^{\,\,3}-3\,\overline{r}^{\,\,}\,%
\overline{\overline{r}}+\overline{\overline{\overline{r}}}\right)$&$105$\\
\hline
\end{tabular}%
\\[1\baselineskip]
\begin{landscape}
\newpage
{\small
\textbf{Table} $\boldsymbol{\varphi.4}$
\hspace{4.5in}
\textbf{Tables} \ $\boldsymbol{\alpha.4}$
}\\
\begin{tabular}{|l | l|l|l|}
\hline
$\genfrac{}{}{0pt}{}{\text{\# roots}}{n}$&$\rho$&$\varphi_{_{D\,\delta\,\rho}}=%
\varphi_{_{4\,0\,\rho}}$&$\sum_{\text{coef}>0}$\\
\hline
$1$&$3$&$-3\,\,{\overline{r}}^{\,\,4}+6\,\,{\overline{r}}^{\,\,2}\,{%
\overline{\overline{r}}}-4\,\,{\overline{r}}\,\,{\overline{\overline{%
\overline{r}}}}-0\,\,{\overline{\overline{r}}}^{\,\,2}+1\,\,{\overline{%
\overline{\overline{\overline{r}}}}}$&$7$\\
\hline
$2$&$2$&$-8\,\,{\overline{r}}^{\,\,4}+16\,\,{\overline{r}}^{\,\,2}\,{%
\overline{\overline{r}}}-4\,\,{\overline{r}}\,\,{\overline{\overline{%
\overline{r}}}}-5\,\,{\overline{\overline{r}}}^{\,\,2}+1\,\,{\overline{%
\overline{\overline{\overline{r}}}}}$&$17$\\
\hline
$3$&$1$&$-9\,\,{\overline{r}}^{\,\,4}+18\,\,{\overline{r}}^{\,\,2}\,{%
\overline{\overline{r}}}-4\,\,{\overline{r}}\,\,{\overline{\overline{%
\overline{r}}}}-6\,\,{\overline{\overline{r}}}^{\,\,2}+1\,\,{\overline{%
\overline{\overline{\overline{r}}}}}$&$19$\\
\hline
$4$&$0$&$0$&$0$\\
\hline
$5$&$-1$&$25\,\,{\overline{r}}^{\,\,4}-50\,{\overline{r}}^{\,\,2}\,{%
\overline{\overline{r}}}+12\,\,{\overline{r}}\,\,{\overline{\overline{%
\overline{r}}}}+16\,\,{\overline{\overline{r}}}^{\,\,2}-3\,\,{\overline{%
\overline{\overline{\overline{r}}}}}$&$53$\\
\hline
$6$&$-2$&$72\,\,{\overline{r}}^{\,\,4}-144\,\,{\overline{r}}^{\,\,2}\,{%
\overline{\overline{r}}}+36\,\,{\overline{r}}\,\,{\overline{\overline{%
\overline{r}}}}+45\,\,{\overline{\overline{r}}}^{\,\,2}-9\,\,{\overline{%
\overline{\overline{\overline{r}}}}}$&$153$\\
\hline
$7$&$-3$&$147\,\,{\overline{r}}^{\,\,4}-294\,\,{\overline{r}}^{\,\,2}\,{%
\overline{\overline{r}}}+76\,\,{\overline{r}}\,\,{\overline{\overline{%
\overline{r}}}}+90\,\,{\overline{\overline{r}}}^{\,\,2}-19\,\,{\overline{%
\overline{\overline{\overline{r}}}}}$&$313$\\
\hline
$8$&$-4$&$256\,\,{\overline{r}}^{\,\,4}-512\,\,{\overline{r}}^{\,\,2}\,{%
\overline{\overline{r}}}+136\,\,{\overline{r}}\,\,{\overline{\overline{%
\overline{r}}}}+154\,\,{\overline{\overline{r}}}^{\,\,2}-34\,\,{\overline{%
\overline{\overline{\overline{r}}}}}$&$546$\\
\hline
$9$&$-5$&$405\,\,{\overline{r}}^{\,\,4}-810\,\,{\overline{r}}^{\,\,2}\,{%
\overline{\overline{r}}}+220\,\,{\overline{r}}\,\,{\overline{\overline{%
\overline{r}}}}+240\,\,{\overline{\overline{r}}}^{\,\,2}-55\,\,{\overline{%
\overline{\overline{\overline{r}}}}}$&$865$\\
\hline
$10$&$-6$&$600\,\,{\overline{r}}^{\,\,4}-1200\,\,{\overline{r}}^{\,\,2}\,{%
\overline{\overline{r}}}+332\,\,{\overline{r}}\,\,{\overline{\overline{%
\overline{r}}}}+351\,\,{\overline{\overline{r}}}^{\,\,2}-83\,\,{\overline{%
\overline{\overline{\overline{r}}}}}$&$1283$\\
\hline
\end{tabular}%
{\setlength{\tabcolsep}{0.063in}%
\begin{footnotesize}\quad
\begin{tabular}{|c|c|c|c|c|c|c|c@{}|}
\hline
$\rho\phantom{\bigg|}$&\multicolumn{7}{c@{}|}{$\alpha_{_{4,0,\rho}}$}\\ 
\hline
$3\phantom{\Big|}$&$615$&&&&&&\\
\hline
$2\phantom{\Big|}$&&$615$&&&&&\\
\hline
$1\phantom{\Big|}$&&&$205$&&&&\\
\hline
$0$&&&&&&&\\
\hline
--$1\phantom{\Big|}$&&&&$123$&&&\\
\hline
--$2\phantom{\Big|}$&&&&&$205$&&\\
\hline
--$3\phantom{\Big|}$&&&&&&$205$&\\
\hline
--$4\phantom{\Big|}$&&&&&&&$615$\\
\hline
--$5\phantom{\Big|}$&--$351$&$64$&$49$&--$55$&--$192$&--$267$&--$848\,$\\
\hline
--$6\phantom{\Big|}$&$240$&--$35$&--$30$&$32$&$105$&$130$&$310$\\
\hline
\end{tabular}
\end{footnotesize}
} 
\\[2.5\baselineskip]
\begin{tabular}[t]{|c|c|c|c|c|c|c|c|c|c|c|c|c|c|c|c|c|c|c|c|c|c|}
\hline
$\rho$&\multicolumn{21}{c|}{$\alpha_{_{4,0,\rho}}$}\\
\hline
$3$&$\boldsymbol{1}$&$1$&&$9$&&&$1$&&&&$67$&&&&&$155$&&&&&\\
\hline
$2$&--$\boldsymbol{6}$&&$2$&&$9$&&&$9$&&&&$134$&&&&&$310$&&&&\\
\hline
$1$&$\boldsymbol{5}$&$8$&$1$&&&$9$&&&$15$&&&&$67$&&&&&$31$&&&\\
\hline
$0$&&&&&&&&&&&&&&&&&&&&&\\
\hline
--$1$&&$3$&$1$&--$45$&&$9$&&&&$45$&&&&$67$&&&&&$155$&&\\
\hline
--$2$&&&&$16$&$1$&--$2$&--$2$&$1$&$8$&--$34$&&&&&$134$&&&&&$310$&\\
\hline
--$3$&&&&&&&$1$&&&$9$&--$77$&$16$&$25$&$41$&--$144$&&&&&&$155$\\
\hline
--$4$&&&&&&&&&&&$45$&--$5$&--$12$&$17$&$45$&--$120$&$35$&$9$&--$80$&--$315$&--%
$195$\\
\hline
--$5$&&&&&&&&&&&&&&&&$77$&--$16$&--$5$&$41$&$144$&$67$\\
\hline
--$6$&&&&&&&&&&&&&&&&&&&&&\\
\hline
\end{tabular}%
%
\newpage
\textbf{Tables} $\boldsymbol{\varphi.5}$\\
\begin{tabular}{|l|l|l|l|}
\hline
$\genfrac{}{}{0pt}{}{\text{\# roots}}{n}$&$\rho$&$\varphi_{_{D\,\delta\,\rho}}=%
\varphi_{_{5\,0\,\rho}}$&$\sum_{\text{coef}>0}$\\
\hline
$1$&$4$& \
$-4\,\,{\overline{r}}^{\,\,5}+10\,\,{\overline{r}}^{\,\,3}\,{\overline{%
\overline{r}}}-10\,\,{\overline{r}}\,^2\,{\overline{\overline{%
\overline{r}}}}-0\,\,{\overline{r}}\,\,{\overline{\overline{r}}}^{\,\,2}+5\,\,%
{\overline{r}}\,\,{\overline{\overline{\overline{\overline{r}}}}}+0\,{%
\overline{\overline{r}}\,}\,{\overline{\overline{\overline{r}}}}-1\,%
\overline{{\overline{\overline{\overline{\overline{r}}}}}}$&$15$\\
\hline
$2$&$3$&$-24\,\,{\overline{r}}^{\,\,5}+60\,\,{\overline{r}}^{\,\,3}\,{%
\overline{\overline{r}}}-20\,\,{\overline{r}}\,^2\,{\overline{\overline{%
\overline{r}}}}-30\,\,{\overline{r}}\,\,{\overline{\overline{r}}}^{\,\,2}+5\,%
\,{\overline{r}}\,\,{\overline{\overline{\overline{\overline{r}}}}}+10\,{%
\overline{\overline{r}}\,}\,{\overline{\overline{\overline{r}}}}-1\,%
\overline{{\overline{\overline{\overline{\overline{r}}}}}}$&$75$\\
\hline
$3$&$2$&$-54\,\,{\overline{r}}^{\,\,5}+135\,\,{\overline{r}}^{\,\,3}\,{%
\overline{\overline{r}}}-35\,\,{\overline{r}}\,^2\,{\overline{\overline{%
\overline{r}}}}-75\,\,\,{\overline{r}}\,\,{\overline{\overline{r}}}^{\,\,2}+5%
\,\,{\overline{r}}\,\,{\overline{\overline{\overline{\overline{r}}}}}+25\,\,{%
\overline{\overline{r}}\,}\,{\overline{\overline{\overline{r}}}}-1\,%
\overline{{\overline{\overline{\overline{\overline{r}}}}}}$&$165$\\
\hline
$4$&$1$&$-64\,\,{\overline{r}}^{\,\,5}+160\,\,{\overline{r}}^{\,\,3}\,{%
\overline{\overline{r}}}-40\,\,{\overline{r}}\,^2\,{\overline{\overline{%
\overline{r}}}}-90\,\,\,{\overline{r}}\,\,{\overline{\overline{r}}}^{\,\,2}+5%
\,\,{\overline{r}}\,\,{\overline{\overline{\overline{\overline{r}}}}}+30\,\,{%
\overline{\overline{r}}\,}\,{\overline{\overline{\overline{r}}}}-1\,%
\overline{{\overline{\overline{\overline{\overline{r}}}}}}$&$195$\\
\hline
$5$&$0$&$0$&$0$\\
\hline
$6$&$-1$&$216\,\,{\overline{r}}^{\,\,5}-540\,\,{\overline{r}}^{\,\,3}\,{%
\overline{\overline{r}}}+140\,\,{\overline{r}}\,^2\,{\overline{\overline{%
\overline{r}}}}+300\,\,\,{\overline{r}}\,\,{\overline{\overline{r}}}^{\,\,%
2}-20\,\,{\overline{r}}\,\,{\overline{\overline{\overline{%
\overline{r}}}}}-100\,{\overline{\overline{r}}\,}\,{\overline{\overline{%
\overline{r}}}}+4\,\overline{{\overline{\overline{\overline{%
\overline{r}}}}}}$&$660$\\
\hline
$7$&$-2$&$686\,\,{\overline{r}}^{\,\,5}-1715\,\,{\overline{r}}^{\,\,3}\,{%
\overline{\overline{r}}}+455\,\,{\overline{r}}\,^2\,{\overline{\overline{%
\overline{r}}}}+945\,\,\,{\overline{r}}\,\,{\overline{\overline{r}}}^{\,\,%
2}-70\,\,{\overline{r}}\,\,{\overline{\overline{\overline{%
\overline{r}}}}}-315\,{\overline{\overline{r}}\,}\,{\overline{\overline{%
\overline{r}}}}+14\,\overline{{\overline{\overline{\overline{%
\overline{r}}}}}}$&$2100$\\
\hline
$8$&$-3$&$1536\,\,{\overline{r}}^{\,\,5}-3840\,\,{\overline{r}}^{\,\,3}\,{%
\overline{\overline{r}}}+1040\,\,{\overline{r}}\,^2\,{\overline{\overline{%
\overline{r}}}}+2100\,\,\,{\overline{r}}\,\,{\overline{\overline{r}}}^{\,\,%
2}-170\,\,{\overline{r}}\,\,{\overline{\overline{\overline{%
\overline{r}}}}}-700\,{\overline{\overline{r}}\,}\,{\overline{\overline{%
\overline{r}}}}+34\,\overline{{\overline{\overline{\overline{%
\overline{r}}}}}}$&$4710$\\
\hline
$9$&$-4$&$2916\,{\overline{r}}^{\,\,5}-7290\,\,{\overline{r}}^{\,\,3}\,{%
\overline{\overline{r}}}+2010\,\,{\overline{r}}\,^2\,{\overline{\overline{%
\overline{r}}}}+3960\,\,\,{\overline{r}}\,\,{\overline{\overline{r}}}^{\,\,%
2}-345\,\,{\overline{r}}\,\,{\overline{\overline{\overline{%
\overline{r}}}}}-1320\,{\overline{\overline{r}}\,}\,{\overline{\overline{%
\overline{r}}}}+69\,\overline{{\overline{\overline{\overline{%
\overline{r}}}}}}$&$8955$\\
\hline
$10$&$-5$&$5000\,{\overline{r}}^{\,\,5}-12500\,\,{\overline{r}}^{\,\,3}\,{%
\overline{\overline{r}}}+3500\,\,{\overline{r}}\,^2\,{\overline{\overline{%
\overline{r}}}}+6750\,\,\,{\overline{r}}\,\,{\overline{\overline{r}}}^{\,\,%
2}-625\,\,{\overline{r}}\,\,{\overline{\overline{\overline{%
\overline{r}}}}}-2250\,{\overline{\overline{r}}\,}\,{\overline{\overline{%
\overline{r}}}}+125\,\overline{{\overline{\overline{\overline{%
\overline{r}}}}}}$&$15375$\\
\hline
\end{tabular}%
\newpage
\textbf{Tables} \ $\boldsymbol{\alpha.5}$ \\
\begin{tabular}[t]{|c|c|c|c|c|c|c|c|c|c|c|c|c|c|c|c|c|c|c|c|c|c|}
\hline
$\rho$&\multicolumn{21}{c|}{$\alpha_{_{5,0,\rho}}$}\\
\hline
$4$&--$\boldsymbol{1}$&$\boldsymbol{1}$&&&&$2$&&&&$28$&&&&&$26$&&&&&&$171$\\
\hline
$3$&$\boldsymbol{2}$&&$\boldsymbol{1}$&&$1$&&&&$28$&&&&&$182$&&&&&&$38$&\\
\hline
$2$&--$\boldsymbol{2}$&--$\boldsymbol{6}$&--$%
\boldsymbol{4}$&$4$&&&&$4$&&&&&$182$&&&&&&$228$&&\\
\hline
$1$&$\boldsymbol{1}$&$\boldsymbol{5}$&$%
\boldsymbol{3}$&&$3$&$10$&$28$&&&&&$182$&&&&&&$342$&&&\\
\hline
$0$&&&&&&&&&&&&&&&&&&&&&\\
\hline
--$1$&&&&$1$&$1$&$3$&$21$&$1$&--$35$&--$63$&$91$&&&&&&$57$&&&&\\
\hline
--$2$&&&&&&&--$4$&&$12$&$20$&$-60$&$64$&$30$&--$72$&--$20$&$684$&&&&&\\
\hline
--$3$&&&&&&&&&&&$14$&--$21$&--$7$&$35$&$9$&--$651$&--$27$&$75$&$27$&--$7$&--$6%
6$\\
\hline
--$4$&&&&&&&&&&&&&&&&$182$&$10$&--$32$&--$10$&$4$&$35$\\
\hline
--$5$&&&&&&&&&&&&&&&&&&&&&\\
\hline
\end{tabular}
\\[1\baselineskip]
\setlength{\tabcolsep}{0.063in}%
\begin{tabular}{|c|c|c|c|c|c|c|c|}
\hline
$\rho\phantom{\bigg|}$&\multicolumn{7}{c|}{$\alpha_{_{5,0,\rho}}$}\\
\hline
$4$&&&&&&&$325$\\
\hline
$3$&&&&&&$975$&\\
\hline
$2$&&&&&$1950$&&\\
\hline
$1$&&&&$13$&&&\\
\hline
$0$&&&&&&&\\
\hline
--$1$&&&$975$&&&&\\
\hline
--$2$&&$650$&&&&&\\
\hline
--$3$&$975$&&&&&&\\
\hline
--$4$&--$1100$&--$525$&--$350$&$2$&$175$&--$300$&--$75$\\
\hline
--$5$&$342$&$217$&$162$&--$1$&--$81$&$63$&$44$\\
\hline
\end{tabular}
%
\newpage
\textbf{Table} \ $\boldsymbol{\varphi.6}$\\
\begin{footnotesize}
\begin{tabular}[t]{|l|l|l|l|}
\hline
$\genfrac{}{}{0pt}{}{\text{\# roots}}{n}$&$\rho$&$\varphi_{_{D\,\delta\,\rho}}=%
\varphi_{_{6\,0\,\rho}}$&$\sum_{\text{coef}>0}$\\
\hline
$1$&$5$&$-5\,\,{\overline{r}}^{\,\,6}+15\,\,{\overline{r}}^{\,\,4}\,{%
\overline{\overline{r}}}-20\,\,{\overline{r}}\,^3\,{\overline{\overline{%
\overline{r}}}}+0\,\,{\overline{r}}^2\,\,{\overline{\overline{r}}}^{\,\,2}+15%
\,\,{\overline{r}}\,^2\,{\overline{\overline{\overline{\overline{r}}}}}+0\,%
\overline{r}\,{\overline{\overline{r}}\,}\,{\overline{\overline{%
\overline{r}}}}-6\,\overline{r}\,\,\overline{{\overline{\overline{\overline{%
\overline{r}}}}}}+0\,\,{\overline{\overline{r}}}\,^3+0\,\,\overline{{%
\overline{r}}}\,{\overline{\overline{\overline{\overline{r}}}}}+0\,\,{%
\overline{\overline{\overline{r}}}}^{\,\,2}+1\,\overline{\overline{{%
\overline{\overline{\overline{\overline{r}}}}}}}$&$31$\\
\hline
$2$&$4$&$-64\,\,{\overline{r}}^{\,\,6}+192\,\,{\overline{r}}^{\,\,4}\,{%
\overline{\overline{r}}}-80\,\,{\overline{r}}\,^3\,{\overline{\overline{%
\overline{r}}}}-132\,\,{\overline{r}}^2\,\,{\overline{\overline{r}}}^{\,\,%
2}+30\,\,{\overline{r}}\,^2\,{\overline{\overline{\overline{%
\overline{r}}}}}+60\,\overline{r}\,{\overline{\overline{r}}\,}\,{\overline{%
\overline{\overline{r}}}}-6\,\overline{r}\,\,\overline{{\overline{\overline{%
\overline{\overline{r}}}}}}+14\,\,{\overline{\overline{r}}}\,^3-15\,\,%
\overline{{\overline{r}}}\,{\overline{\overline{\overline{\overline{r}}}}}+0\,%
\,{\overline{\overline{\overline{r}}}}^{\,\,2}+1\,\overline{\overline{{%
\overline{\overline{\overline{\overline{r}}}}}}}$&$297$\\
\hline
$3$&$3$&$-243\,\,{\overline{r}}^{\,\,6}+729\,\,{\overline{r}}^{\,\,4}\,{%
\overline{\overline{r}}}-216\,\,{\overline{r}}\,^3\,{\overline{\overline{%
\overline{r}}}}-567\,\,{\overline{r}}^2\,\,{\overline{\overline{r}}}^{\,\,%
2}+45\,\,{\overline{r}}\,^2\,{\overline{\overline{\overline{%
\overline{r}}}}}+234\,\overline{r}\,{\overline{\overline{r}}\,}\,{\overline{%
\overline{\overline{r}}}}-6\,\overline{r}\,\,\overline{{\overline{\overline{%
\overline{\overline{r}}}}}}+72\,\,{\overline{\overline{r}}}\,^3-30\,\,%
\overline{{\overline{r}}}\,{\overline{\overline{\overline{\overline{r}}}}}-19%
\,\,{\overline{\overline{\overline{r}}}}^{\,\,2}+1\,\overline{\overline{{%
\overline{\overline{\overline{\overline{r}}}}}}}$&$1081$\\
\hline
$4$&$2$&$-512\,\,{\overline{r}}^{\,\,6}+1536\,\,{\overline{r}}^{\,\,4}\,{%
\overline{\overline{r}}}-416\,\,{\overline{r}}\,^3\,{\overline{\overline{%
\overline{r}}}}-1224\,\,{\overline{r}}^2\,\,{\overline{\overline{r}}}^{\,\,%
2}+66\,\,{\overline{r}}\,^2\,{\overline{\overline{\overline{%
\overline{r}}}}}+492\,\overline{r}\,{\overline{\overline{r}}\,}\,{\overline{%
\overline{\overline{r}}}}-6\,\overline{r}\,\,\overline{{\overline{\overline{%
\overline{\overline{r}}}}}}+162\,\,{\overline{\overline{r}}}\,^3-51\,\,%
\overline{{\overline{r}}}\,{\overline{\overline{\overline{\overline{r}}}}}-48%
\,\,{\overline{\overline{\overline{r}}}}^{\,\,2}+1\,\overline{\overline{{%
\overline{\overline{\overline{\overline{r}}}}}}}$&$2257$\\
\hline
$5$&$1$&$-625\,\,{\overline{r}}^{\,\,6}+1875\,\,{\overline{r}}^{\,\,4}\,{%
\overline{\overline{r}}}-500\,\,{\overline{r}}\,^3\,{\overline{\overline{%
\overline{r}}}}-1500\,\,{\overline{r}}^2\,\,{\overline{\overline{r}}}^{\,\,%
2}+75\,\,{\overline{r}}\,^2\,{\overline{\overline{\overline{%
\overline{r}}}}}+600\,\overline{r}\,{\overline{\overline{r}}\,}\,{\overline{%
\overline{\overline{r}}}}-6\,\overline{r}\,\,\overline{{\overline{\overline{%
\overline{\overline{r}}}}}}+200\,\,{\overline{\overline{r}}}\,^3-60\,\,%
\overline{{\overline{r}}}\,{\overline{\overline{\overline{\overline{r}}}}}-60%
\,\,{\overline{\overline{\overline{r}}}}^{\,\,2}+1\,\overline{\overline{{%
\overline{\overline{\overline{\overline{r}}}}}}}$&$2751$\\
\hline
$6$&$0$&$0$&$0$\\
\hline
$7$&$-1$&$2401\,\,{\overline{r}}^{\,\,6}-7203\,\,{\overline{r}}^{\,\,4}\,{%
\overline{\overline{r}}}+1960\,\,{\overline{r}}\,^3\,{\overline{\overline{%
\overline{r}}}}+5733\,\,{\overline{r}}^2\,\,{\overline{\overline{r}}}^{\,\,%
2}-315\,\,{\overline{r}}\,^2\,{\overline{\overline{\overline{%
\overline{r}}}}}-2310\,\overline{r}\,{\overline{\overline{r}}\,}\,{\overline{%
\overline{\overline{r}}}}+30\,\overline{r}\,\,\overline{{\overline{\overline{%
\overline{\overline{r}}}}}}-756\,\,{\overline{\overline{r}}}\,^3+240\,\,%
\overline{{\overline{r}}}\,{\overline{\overline{\overline{%
\overline{r}}}}}+225\,\,{\overline{\overline{\overline{r}}}}^{\,\,2}-5\,%
\overline{\overline{{\overline{\overline{\overline{%
\overline{r}}}}}}}$&$10589$\\
\hline
$8$&$-2$&$8192\,\,{\overline{r}}^{\,\,6}-24576\,\,{\overline{r}}^{\,\,4}\,{%
\overline{\overline{r}}}+6784\,\,{\overline{r}}\,^3\,{\overline{\overline{%
\overline{r}}}}+19488\,\,{\overline{r}}^2\,\,{\overline{\overline{r}}}^{\,\,%
2}-1140\,\,{\overline{r}}\,^2\,{\overline{\overline{\overline{%
\overline{r}}}}}-7896\,\overline{r}\,{\overline{\overline{r}}\,}\,{\overline{%
\overline{\overline{r}}}}+120\,\overline{r}\,\,\overline{{\overline{%
\overline{\overline{\overline{r}}}}}}-2548\,\,{\overline{\overline{r}}}\,%
^3+840\,\,\overline{{\overline{r}}}\,{\overline{\overline{\overline{%
\overline{r}}}}}+756\,\,{\overline{\overline{\overline{r}}}}^{\,\,2}-20\,%
\overline{\overline{{\overline{\overline{\overline{%
\overline{r}}}}}}}$&$36180$\\
\hline
$9$&$-3$&$19683\,\,{\overline{r}}^{\,\,6}-59049\,\,{\overline{r}}^{\,\,4}\,{%
\overline{\overline{r}}}+16524\,\,{\overline{r}}\,^3\,{\overline{\overline{%
\overline{r}}}}+46656\,\,{\overline{r}}^2\,\,{\overline{\overline{r}}}^{\,\,%
2}-2889\,\,{\overline{r}}\,^2\,{\overline{\overline{\overline{%
\overline{r}}}}}-19008\,\overline{r}\,{\overline{\overline{r}}\,}\,{%
\overline{\overline{\overline{r}}}}+330\,\overline{r}\,\,\overline{{%
\overline{\overline{\overline{\overline{r}}}}}}-6048\,\,{\overline{%
\overline{r}}}\,^3+2064\,\,\overline{{\overline{r}}}\,{\overline{\overline{%
\overline{\overline{r}}}}}+1792\,\,{\overline{\overline{\overline{r}}}}^{\,\,%
2}-55\,\overline{\overline{{\overline{\overline{\overline{%
\overline{r}}}}}}}$&$87049$\\
\hline
$10$&$-4$&$40000\,\,{\overline{r}}^{\,\,6}-120000\,\,{\overline{r}}^{\,\,4}\,%
{\overline{\overline{r}}}+34000\,\,{\overline{r}}\,^3\,{\overline{\overline{%
\overline{r}}}}+94500\,\,{\overline{r}}^2\,\,{\overline{\overline{r}}}^{\,\,%
2}-6150\,\,{\overline{r}}\,^2\,{\overline{\overline{\overline{%
\overline{r}}}}}-38700\,\overline{r}\,{\overline{\overline{r}}\,}\,{%
\overline{\overline{\overline{r}}}}+750\,\overline{r}\,\,\overline{{%
\overline{\overline{\overline{\overline{r}}}}}}-12150\,\,{\overline{%
\overline{r}}}\,^3+4275\,\,\overline{{\overline{r}}}\,{\overline{\overline{%
\overline{\overline{r}}}}}+3600\,\,{\overline{\overline{\overline{r}}}}^{\,\,%
2}-125\,\overline{\overline{{\overline{\overline{\overline{%
\overline{r}}}}}}}$&$177125$\\
\hline
\end{tabular}%
\end{footnotesize}
%
%
\\[1\baselineskip]
\textbf{Table} \ $\boldsymbol{\alpha.6}$\\
\begin{tabular}{|c|c|c|c|c|c|c|c|c|c|c|c|c|c|c|c|}
\hline
$\rho$&\multicolumn{15}{c|}{$\alpha_{_{6,0,\rho}}$}\\
\hline
$5$&$\boldsymbol{3}$&$1$&&$1$&&&$3$&&&&$125$&&&&\\
\hline
$4$&--$\boldsymbol{20}$&&$2$&&$1$&&&$4$&&&&$125$&&&\\
\hline
$3$&$\boldsymbol{60}$&--$30$&--$15$&&&$3$&&&$2$&&&&$125$&&\\
\hline
$2$&--$\boldsymbol{120}$&$160$&$60$&--$140$&--$45$&--$30$&&&&$12$&&&&$125$&\\
\hline
$1$&$%
\boldsymbol{77}$&--$81$&--$32$&--$161$&--$36$&--$8$&$1372$&$651$&$83$&$53$&&&&%
&$125$\\
\hline
$0$&&&&&&--$15$&&&&&&&&&\\
\hline
--$1$&&--$10$&$3$&--$140$&--$36$&$2$&$1680$&$756$&$90$&$60$&$67200$&--$25200$&%
--$6825$&--$700$&$300$\\
\hline
--$2$&&&&$20$&$5$&&--$640$&--$280$&--$32$&--$20$&$64800$&$23800$&$6300$&$675$&%
--$200$\\
\hline
--$3$&&&&&&&$105$&$45$&$5$&$3$&--$29925$&--$10800$&--$2800$&--$300$&$75$\\
\hline
--$4$&&&&&&&&&&&$5488$&$1953$&$498$&$53$&--$12$\\
\hline
\end{tabular}%
\\[1\baselineskip]
\textbf{Table} \ $\boldsymbol{\varphi.7}$\\[-1.060\baselineskip]
\begin{Large}
\begin{tabular}[t]{|l|l|l|l|}
\hline
{\normalsize$n$}&{\normalsize$\rho$}&{\normalsize$\varphi_{_{D\,\delta\,\rho}}=\varphi_{_{7\,0\,\rho}}$}&{\normalsize$\sum_{%
\text{coef}>0}$}\\
\hline
$1$&$6$&$\genfrac{}{}{0pt}{}{-6\,\,{\overline{r}}^{\,\,7}\,+\,21\,\,{\overline{r}}^{\,\,5}\,{\overline{%
\overline{r}}}\,-\,35\,\,{\overline{r}}\,^4\,{\overline{\overline{%
\overline{r}}}}\,+\,0\,\,{\overline{r}}^{\,3}\,\,{\overline{\overline{r}}}^{\,\,%
2}\,+\,35\,\,{\overline{r}}\,^3\,{\overline{\overline{\overline{%
\overline{r}}}}}\,+\,0\,\overline{r}^{\,2}\,{\overline{\overline{r}}\,}\,{%
\overline{\overline{\overline{r}}}}\,-\,21\,\overline{r}^{\,2}\,\,\overline{{%
\overline{\overline{\overline{\overline{r}}}}}}\,+\,}{0\,\overline{r}\,{\overline{%
\overline{r}}}\,^3\,+\,0\,\overline{r}\,\overline{{\overline{r}}}\,{\overline{%
\overline{\overline{\overline{r}}}}}\,+\,0\,\overline{r}\,\,{\overline{\overline{%
\overline{r}}}}^{\,\,2}\,+\,7\,\overline{r}\,\overline{\overline{{\overline{%
\overline{\overline{\overline{r}}}}}}}\,+\,0\,\,\overline{{\overline{r}}}^{\,2}\,%
{\overline{\overline{\overline{{r}}}}}\,-\,0\,\overline{\overline{r}}\,\,%
\overline{{\overline{\overline{\overline{\overline{r}}}}}}\,+\,0\,\,\overline{{%
\overline{\overline{r}}}}\,{\overline{\overline{\overline{\overline{r}}}}}\,-\,1\,%
\overline{\overline{{\overline{\overline{\overline{\overline{%
\overline{r}}}}}}}}}$&$63$\\
\hline
$2$&$5$&$\genfrac{}{}{0pt}{}{-160\,\,{\overline{r}}^{\,\,7}\,+\,560\,\,{\overline{r}}^{\,\,5}\,{%
\overline{\overline{r}}}\,-\,280\,\,{\overline{r}}\,^4\,{\overline{\overline{%
\overline{r}}}}\,-\,490\,\,{\overline{r}}^{\,3}\,\,{\overline{\overline{r}}}^{\,\,%
2}\,+\,140\,\,{\overline{r}}\,^3\,{\overline{\overline{\overline{%
\overline{r}}}}}\,+\,280\,\overline{r}^{\,2}\,{\overline{\overline{r}}\,}\,{%
\overline{\overline{\overline{r}}}}\,-\,42\,\overline{r}^{\,2}\,\,\overline{{%
\overline{\overline{\overline{\overline{r}}}}}}\,+\,}{105\,\overline{r}\,{%
\overline{\overline{r}}}\,^3\,-\,105\,\overline{r}\,\overline{{\overline{r}}}\,{%
\overline{\overline{\overline{\overline{r}}}}}\,+\,0\,\overline{r}\,\,{\overline{%
\overline{\overline{r}}}}^{\,\,2}\,+\,7\,\overline{r}\,\overline{\overline{{%
\overline{\overline{\overline{\overline{r}}}}}}}\,-\,35\,\,\overline{{%
\overline{r}}}^{\,2}\,{\overline{\overline{\overline{{r}}}}}\,+\,21\,%
\overline{\overline{r}}\,\,\overline{{\overline{\overline{\overline{%
\overline{r}}}}}}\,+\,0\,\,\overline{{\overline{\overline{r}}}}\,{\overline{%
\overline{\overline{\overline{r}}}}}\,-\,1\,\overline{\overline{{\overline{%
\overline{\overline{\overline{\overline{r}}}}}}}}}$&$1113$\\
\hline
$3$&$4$&$\genfrac{}{}{0pt}{}{-972\,\,{\overline{r}}^{\,\,7}\,+\,3402\,\,{\overline{r}}^{\,\,5}\,{%
\overline{\overline{r}}}\,-\,1134\,\,{\overline{r}}\,^4\,{\overline{\overline{%
\overline{r}}}}\,-\,3402\,\,{\overline{r}}^{\,3}\,\,{\overline{\overline{r}}}^{\,%
\,2}\,+\,315\,\,{\overline{r}}\,^3\,{\overline{\overline{\overline{%
\overline{r}}}}}\,+\,1638\,\overline{r}^{\,2}\,{\overline{\overline{r}}\,}\,{%
\overline{\overline{\overline{r}}}}\,-\,63\,\overline{r}^{\,2}\,\,\overline{{%
\overline{\overline{\overline{\overline{r}}}}}}\,+\,}{882\,\overline{r}\,{%
\overline{\overline{r}}}\,^3\,-\,315\,\overline{r}\,\overline{{\overline{r}}}\,{%
\overline{\overline{\overline{\overline{r}}}}}\,-\,140\,\overline{r}\,\,{%
\overline{\overline{\overline{r}}}}^{\,\,2}\,+\,7\,\overline{r}\,\overline{%
\overline{{\overline{\overline{\overline{\overline{r}}}}}}}\,-\,294\,\,%
\overline{{\overline{r}}}^{\,2}\,{\overline{\overline{\overline{%
{r}}}}}\,+\,42\,\overline{\overline{r}}\,\,\overline{{\overline{%
\overline{\overline{\overline{r}}}}}}\,+\,35\,\,\overline{{\overline{%
\overline{r}}}}\,{\overline{\overline{\overline{\overline{r}}}}}\,-\,1\,%
\overline{\overline{{\overline{\overline{\overline{\overline{%
\overline{r}}}}}}}}}$&$6321$\\
\hline
$4$&$3$&$\genfrac{}{}{0pt}{}{-3072\,\,{\overline{r}}^{\,\,7}\,+\,10752\,\,{\overline{r}}^{\,\,5}\,{%
\overline{\overline{r}}}\,-\,3136\,\,{\overline{r}}\,^4\,{\overline{\overline{%
\overline{r}}}}\,-\,11088\,\,{\overline{r}}^{\,3}\,\,{\overline{\overline{r}}}^{\,%
\,2}\,+\,616\,\,{\overline{r}}\,^3\,{\overline{\overline{\overline{%
\overline{r}}}}}\,+\,5040\,\overline{r}^{\,2}\,{\overline{\overline{r}}\,}\,{%
\overline{\overline{\overline{r}}}}\,-\,84\,\overline{r}^{\,2}\,\,\overline{{%
\overline{\overline{\overline{\overline{r}}}}}}\,+\,}{3024\,\overline{r}\,{%
\overline{\overline{r}}}\,^3\,-\,714\,\overline{r}\,\overline{{\overline{r}}}\,{%
\overline{\overline{\overline{\overline{r}}}}}\,-\,532\,\overline{r}\,\,{%
\overline{\overline{\overline{r}}}}^{\,\,2}\,+\,7\,\overline{r}\,\overline{%
\overline{{\overline{\overline{\overline{\overline{r}}}}}}}\,-\,1008\,\overline{{%
\overline{r}}}^{\,2}\,{\overline{\overline{\overline{{r}}}}}\,+\,63\,%
\overline{\overline{r}}\,\,\overline{{\overline{\overline{\overline{%
\overline{r}}}}}}\,+\,133\,\,\overline{{\overline{\overline{r}}}}\,{\overline{%
\overline{\overline{\overline{r}}}}}\,-\,1\,\overline{\overline{{\overline{%
\overline{\overline{\overline{\overline{r}}}}}}}}}$&$19635$\\
\hline
$5$&$2$&$\genfrac{}{}{0pt}{}{-6250\,\,{\overline{r}}^{\,\,7}\,+\,21875\,\,{\overline{r}}^{\,\,5}\,{%
\overline{\overline{r}}}\,-\,6125\,\,{\overline{r}}\,^4\,{\overline{\overline{%
\overline{r}}}}\,-\,22750\,\,{\overline{r}}^{\,3}\,\,{\overline{\overline{r}}}^{\,%
\,2}\,+\,1050\,\,{\overline{r}}\,^3\,{\overline{\overline{\overline{%
\overline{r}}}}}\,+\,10150\,\overline{r}^{\,2}\,{\overline{\overline{r}}\,}\,{%
\overline{\overline{\overline{r}}}}\,-\,112\,\overline{r}^{\,2}\,\,\overline{{%
\overline{\overline{\overline{\overline{r}}}}}}\,+\,}{6300\,\overline{r}\,{%
\overline{\overline{r}}}\,^3\,-\,1295\,\overline{r}\,\overline{{\overline{r}}}\,{%
\overline{\overline{\overline{\overline{r}}}}}\,-\,1120\,\overline{r}\,\,{%
\overline{\overline{\overline{r}}}}^{\,\,2}\,+\,7\,\overline{r}\,\overline{%
\overline{{\overline{\overline{\overline{\overline{r}}}}}}}\,-\,2100\,\,%
\overline{{\overline{r}}}^{\,2}\,{\overline{\overline{\overline{%
{r}}}}}\,+\,91\,\overline{\overline{r}}\,\,\overline{{\overline{%
\overline{\overline{\overline{r}}}}}}\,+\,280\,\,\overline{{\overline{%
\overline{r}}}}\,{\overline{\overline{\overline{\overline{r}}}}}\,-\,1\,%
\overline{\overline{{\overline{\overline{\overline{\overline{%
\overline{r}}}}}}}}}$&$39753$\\
\hline
$6$&$1$&$\genfrac{}{}{0pt}{}{-7776\,\,{\overline{r}}^{\,\,7}\,+\,27216\,\,{\overline{r}}^{\,\,5}\,{%
\overline{\overline{r}}}\,-\,7560\,\,{\overline{r}}\,^4\,{\overline{\overline{%
\overline{r}}}}\,-\,28350\,\,{\overline{r}}^{\,3}\,\,{\overline{\overline{r}}}^{\,%
\,2}\,+\,1260\,\,{\overline{r}}\,^3\,{\overline{\overline{\overline{%
\overline{r}}}}}\,+\,12600\,\overline{r}^{\,2}\,{\overline{\overline{r}}\,}\,{%
\overline{\overline{\overline{r}}}}\,-\,126\,\overline{r}^{\,2}\,\,\overline{{%
\overline{\overline{\overline{\overline{r}}}}}}\,+\,}{7875\,\overline{r}\,{%
\overline{\overline{r}}}\,^3\,-\,1575\,\overline{r}\,\overline{{\overline{r}}}\,{%
\overline{\overline{\overline{\overline{r}}}}}\,-\,1400\,\overline{r}\,\,{%
\overline{\overline{\overline{r}}}}^{\,\,2}\,+\,7\,\overline{r}\,\overline{%
\overline{{\overline{\overline{\overline{\overline{r}}}}}}}\,-\,2625\,\,%
\overline{{\overline{r}}}^{\,2}\,{\overline{\overline{\overline{%
{r}}}}}\,+\,105\,\overline{\overline{r}}\,\,\overline{{\overline{%
\overline{\overline{\overline{r}}}}}}\,+\,350\,\,\overline{{\overline{%
\overline{r}}}}\,{\overline{\overline{\overline{\overline{r}}}}}\,-\,1\,%
\overline{\overline{{\overline{\overline{\overline{\overline{%
\overline{r}}}}}}}}}$&$49413$\\
\hline
{\small $7$}&{\small $0$}&{\small $0$}&{\small $0$} \\ 
\hline
$8$&$-1$&$\genfrac{}{}{0pt}{}{32768\,\,{\overline{r}}^{\,\,7}\,-\,114688\,\,{\overline{r}}^{\,\,5}\,{%
\overline{\overline{r}}}\,+\,32256\,\,{\overline{r}}\,^4\,{\overline{\overline{%
\overline{r}}}}\,+\,11968\,\,{\overline{r}}^{\,3}\,\,{\overline{\overline{r}}}^{\,%
\,2}\,-\,5600\,\,{\overline{r}}\,^3\,{\overline{\overline{\overline{%
\overline{r}}}}}\,-\,53312\,\overline{r}^{\,2}\,{\overline{\overline{r}}\,}\,{%
\overline{\overline{\overline{r}}}}\,+\,616\,\overline{r}^{\,2}\,\,\overline{{%
\overline{\overline{\overline{\overline{r}}}}}}\,-\,}{32928\,\overline{r}\,{%
\overline{\overline{r}}}\,^3\,+\,6860\,\overline{r}\,\overline{{\overline{r}}}\,{%
\overline{\overline{\overline{\overline{r}}}}}\,+\,5880\,\overline{r}\,\,{%
\overline{\overline{\overline{r}}}}^{\,\,2}\,-\,42\,\overline{r}\,\overline{%
\overline{{\overline{\overline{\overline{\overline{r}}}}}}}\,+\,10976\,\,%
\overline{{\overline{r}}}^{\,2}\,{\overline{\overline{\overline{%
{r}}}}}\,-\,490\,\overline{\overline{r}}\,\,\overline{{\overline{%
\overline{\overline{\overline{r}}}}}}\,-\,1470\,\,\overline{{\overline{%
\overline{r}}}}\,{\overline{\overline{\overline{\overline{r}}}}}\,+\,6\,%
\overline{\overline{{\overline{\overline{\overline{\overline{%
\overline{r}}}}}}}}}$&$208530$\\
\hline
$9$&$-2$&$\genfrac{}{}{0pt}{}{118098\,\,{\overline{r}}^{\,\,7}\,-\,413343\,\,{\overline{r}}^{\,\,5}\,{%
\overline{\overline{r}}}\,+\,117369\,\,{\overline{r}}\,^4\,{\overline{\overline{%
\overline{r}}}}\,+\,428652\,\,{\overline{r}}^{\,3}\,\,{\overline{\overline{r}}}^{%
\,\,2}\,-\,20979\,\,{\overline{r}}\,^3\,{\overline{\overline{\overline{%
\overline{r}}}}}\,-\,192780\,\overline{r}^{\,2}\,{\overline{\overline{r}}\,}\,{%
\overline{\overline{\overline{r}}}}\,+\,2457\,\overline{r}^{\,2}\,\,\overline{{%
\overline{\overline{\overline{\overline{r}}}}}}\,-}{117936\,\overline{r}\,{%
\overline{\overline{r}}}\,^3\,+\,25326\,\overline{r}\,\overline{{\overline{r}}}\,%
{\overline{\overline{\overline{\overline{r}}}}}\,+\,21168\,\overline{r}\,\,{%
\overline{\overline{\overline{r}}}}^{\,\,2}\,-\,189\,\overline{r}\,\overline{%
\overline{{\overline{\overline{\overline{\overline{r}}}}}}}\,+\,39312\,%
\overline{{\overline{r}}}^{\,2}\,{\overline{\overline{\overline{%
{r}}}}}\,-\,1890\,\overline{\overline{r}}\,\,\overline{{\overline{%
\overline{\overline{\overline{r}}}}}}\,-\,5292\,\,\overline{{\overline{%
\overline{r}}}}\,{\overline{\overline{\overline{\overline{r}}}}}\,+\,27\,%
\overline{\overline{{\overline{\overline{\overline{\overline{%
\overline{r}}}}}}}}}$&$752409$\\
\hline
$10$&$-3$&$\genfrac{}{}{0pt}{}{300000\,\,{\overline{r}}^{\,\,7}\,-\,1050000\,\,{\overline{r}}^{\,\,5}\,{%
\overline{\overline{r}}}\,+\,301000\,\,{\overline{r}}\,^4\,{\overline{\overline{%
\overline{r}}}}\,+\,1086750\,\,{\overline{r}}^{\,3}\,\,{\overline{%
\overline{r}}}^{\,\,2}\,-\,55300\,\,{\overline{r}}\,^3\,{\overline{\overline{%
\overline{\overline{r}}}}}\,-\,491400\,\overline{r}^{\,2}\,{\overline{%
\overline{r}}\,}\,{\overline{\overline{\overline{r}}}}\,+\,6846\,\overline{r}^{\,%
2}\,\,\overline{{\overline{\overline{\overline{\overline{r}}}}}}\,-}{297675\,%
\overline{r}\,{\overline{\overline{r}}}\,^3\,+\,65385\,\overline{r}\,\overline{{%
\overline{r}}}\,{\overline{\overline{\overline{\overline{r}}}}}\,+\,53760\,%
\overline{r}\,\,{\overline{\overline{\overline{r}}}}^{\,\,2}\,-\,581\,%
\overline{r}\,\overline{\overline{{\overline{\overline{\overline{%
\overline{r}}}}}}}\,+\,99225\,\,\overline{{\overline{r}}}^{\,2}\,{\overline{%
\overline{\overline{{r}}}}}\,-\,5103\,\overline{\overline{r}}\,\,%
\overline{{\overline{\overline{\overline{\overline{r}}}}}}\,-\,13440\,\,%
\overline{{\overline{\overline{r}}}}\,{\overline{\overline{\overline{%
\overline{r}}}}}\,+\,83\,\overline{\overline{{\overline{\overline{\overline{%
\overline{\overline{r}}}}}}}}}$&$1913499$\\
\hline
\end{tabular}%
\end{Large}
\end{landscape}

\textbf{Sub-Supportive Normalized Girard-Waring Results}
\\ \noindent
To present our results, let
$\mathcal{S}:=\left\{s_{_1},\,s_{_2}\right\}=\mathcal{R}^{\left(D\,-\,2%
\right)}_{_f}$ where $D$ is the degree of $f$;
let$\,\,\mathcal{T}:=\left\{t_{_1},\,t_{_2},\,t_{_3}\right\}=\mathcal{R}^{%
\left(D\,-\,3\right)}_{_f};$ let \
$\mathcal{U}:=\mathcal{R}^{\left(D\,-\,4\right)}_{_f}$ etc. \ \ We have the
following results.  \qquad \\ \textbf{N.B.:}  The coefficients in Table $\overline{\text{GW}}\boldsymbol{.}\boldsymbol{\mathfrak{n}}$=2 are exactly the coefficients of Tchebyshev polynomials of the first kind.  Thus these tables generalize Tchebyshev polynomials. \\[1\baselineskip]
\textbf{Table \ $\overline{\text{GW}}\boldsymbol{.}\boldsymbol{\mathfrak{n}}$=2} \\[-0.5\baselineskip]
\begin{tabular}[t]{|l|l|}
\hline
{\setlength{\tabcolsep}{0pt}\begin{tabular}{l}
Sum of Positive\\
Coefficients
\end{tabular}}%
&Data Family \ $\mathcal{F}_{_2}\equiv\mathcal{S}=\left\{s_1,\,s_2\right\}$\\
\hline
$1:$&$\overline{s^1}\,=\,\overline{s}^{\,\,1}$\\
\hline
$2:$&$\overline{s^2}\,=\,2\,\,\overline{s}^{\,\,2}-1\,\overline{%
\overline{s}}$\\
\hline
$4:$&$\overline{s^3}\,=\,{4\,\,\overline{s}^{\,\,3}-3\,\overline{s}\,\,%
\overline{\overline{s}}}$\\
\hline
$9:$&$\overline{s^4}\,=\,8\,\overline{s}^{\,\,4}-8\,\overline{s}^{\,2}\,%
\overline{\overline{s}}\,+1\,\overline{\overline{s}}^{\,\,2}$\\
\hline
$21:$&$\overline{s^5}\,=\,16\,\overline{s}^{\,\,5}-20\,\overline{s}^{\,\,3}\,%
\overline{\overline{s}}\,+5\,\overline{s}\,\overline{\overline{s}}^{\,\,2}$\\
\hline
$50:$&$\overline{s^6}\,=\,32\,\overline{s}^{\,\,6}-48\,\overline{s}^{\,\,4}\,%
\overline{\overline{s}}\,\,+18\,\overline{s}^{\,2}\,\,\overline{%
\overline{s}}^{\,\,2}\,-\overline{\overline{s}}^{\,3}$\\
\hline
$120:$&$\overline{s^7}\,=\,64\,\overline{s}^{\,\,7}\,-\,112\,\overline{s}^{\,%
\,5}\,\overline{\overline{s}}\,+56\,\overline{s}^{\,3}\,\overline{%
\overline{s}}^{\,\,2}\,-\,7\,\overline{s}\,\overline{\overline{s}}^{\,\,3}$\\
\hline
$289:$&$\overline{s^8}\,=\,128\,\overline{s}^{\,\,8}\,-\,256\,\overline{s}^{\,%
\,6}\,\overline{\overline{s}}\,+\,160\,\overline{s}^{\,4}\,\overline{%
\overline{s}}^{\,\,2}\,-\,32\,\overline{s}^{\,\,2}\,\overline{\overline{s}}^{%
\,\,3}\,+1\,\overline{\overline{s}}^{\ 4}$\\
\hline
\end{tabular}%
\\[1\baselineskip]
\textbf{Table \ $\overline{\text{GW}}\boldsymbol{.}\boldsymbol{\mathfrak{n}}$=3} \\[-0.5\baselineskip]
\begin{tabular}[t]{|l|l|}
\hline
{\setlength{\tabcolsep}{0pt}\begin{tabular}{l}
Sum of Positive\\
Coefficients
\end{tabular}}%
&Data Family \
$\mathcal{F}_{_3}\equiv\mathcal{T}=\left\{t_1,\,t_2,\,t_3\right\}$\\
\hline
$1:$&$\overline{t^1}\,=\,\overline{t}^{\,\,1}$\\
\hline
$3:$&$\overline{t^2}\,=\,3\,\,\overline{t}^{\,\,2}-2\,\overline{%
\overline{t}}$\\
\hline
$10:$&$\overline{t^3}\,=\,9\,\,\overline{t}^{\,\,3}-9\,\overline{t}\,\,%
\overline{\overline{t}}\,+\overline{\overline{\overline{t}}}$\\
\hline
$37:$&$\overline{t^4}\,=\,27\,\overline{t}^{\,\,4}-36\,\overline{t}^{\,2}\,%
\overline{\overline{t}}\,+\,4\,\overline{t}\,\overline{\overline{%
\overline{t}}}\,\,+\,6\,\overline{\overline{t}}^{\,\,2}$\\
\hline
$141:$&$\overline{t^5}\,=\,81\,\overline{t}^{\,\,5}-135\,\overline{t}^{\,\,3}%
\,\overline{\overline{t}}\,\,+15\,\overline{t}^{\,\,2}\,\overline{\overline{%
\overline{t}}}\,\,+45\,\overline{t}\,\overline{\overline{t}}^{\,\,2}\,-5\,%
\overline{\overline{t}}\,\overline{\overline{\overline{t}}}$\\
\hline
$541:$&%
\begin{tabular}{l}
$\overline{t^6}\,=\,243\,\overline{t}^{\,\,6}-486\,\overline{t}^{\,\,4}\,%
\overline{\overline{t}}\,\,+54\,\overline{t}^{\,3}\,\,\overline{\overline{%
\overline{t}}}\,+243\,\overline{t}^{\,2}\,\,\overline{\overline{t}}^{\,\,2}\,%
-$\\
$\,\,\,\,\,\,\,\,\,\,\,\,\,\,\,\,\,\,\,\,\,\,\,\,\,\,36\,\overline{t}\,\,%
\overline{\overline{t}}\,\,\overline{\overline{\overline{t}}}\,\,-18\,%
\overline{\overline{t}}^{\,\,3}
+1\,\overline{\overline{\overline{t}}}^{\,\,2}$
\end{tabular}%
$\,\,$\\
\hline
$2080:$&%
\begin{tabular}{l}
$\overline{t^7}\,=\,729\,\overline{t}^{\,\,7}\,-\,1701\,\overline{t}^{\,\,5}%
\,\overline{\overline{t}}\,+189\,\overline{t}^{\,\,4}\,\overline{\overline{%
\overline{t}}}\,+1134\,\overline{t}^{\,\,3}\,\overline{\overline{t}}^{\,\,%
2}-$\\
$\,\,\,\,\,\,\,\,\,\,\,\,\,\,\,\,\,\,\,\,\,\,\,\,\,\,\,\,189\,\overline{t}^{\,%
\,2}\,\,\overline{\overline{t}}\,\,\overline{\overline{\overline{t}}}\,-\,189%
\,\overline{t}\,\,\overline{\overline{t}}^{\,\,3}+7\,\overline{t}\,\,%
\overline{\overline{\overline{t}}}^{\,\,2}+\,21\,\overline{\overline{t}}^{\,\,%
2}\,\,\overline{\overline{\overline{t}}}$
\end{tabular}%
$\,$\\
\hline
\end{tabular}%
\\[1\baselineskip]
\newpage
\noindent
\textbf{Table \ $\overline{\text{GW}}\boldsymbol{.}\boldsymbol{\mathfrak{n}}$=4} \\[-0.5\baselineskip]
\begin{tabular}[t]{|l|l|}
\hline
{\setlength{\tabcolsep}{0pt}\begin{tabular}{l}
Sum of Positive\\
Coefficients
\end{tabular}}%
&Data Family \
$\mathcal{F}_{_4}\equiv\mathcal{U}=\left\{u_1,\,u_2,\,u_3,\,u_4\right\}$\\
\hline
$1:$&$\overline{u^1}\,=\,\overline{u}^{\,\,1}$\\
\hline
$4:$&$\overline{u^2}\,=\,4\,\,\overline{u}^{\,\,2}-3\,\overline{%
\overline{u}}$\\
\hline
$19:$&$\overline{u^3}\,=\,16\,\,\overline{u}^{\,\,3}-18\,\overline{u}\,\,%
\overline{\overline{u}}\,+3\,\overline{\overline{\overline{u}}}$\\
\hline
$98:$&$\overline{u^4}\,=\,64\,\overline{u}^{\,\,4}-96\,\overline{u}^{\,2}\,%
\overline{\overline{u}}\,+\,16\,\overline{u}\,\overline{\overline{%
\overline{u}}}\,\,+\,18\,\overline{\overline{u}}^{\,\,2}\,-\,\overline{%
\overline{\overline{\overline{u}}}}$\\
\hline
$516:$&%
{\setlength{\tabcolsep}{0pt}\begin{tabular}{l}
$\overline{u^5}\,=\,256\,\overline{u}^{\,\,5}-480\,\overline{u}^{\,\,3}\,%
\overline{\overline{u}}\,\,+80\,\overline{u}^{\,\,2}\,\overline{\overline{%
\overline{u}}}\,\,+180\,\overline{u}\,\overline{\overline{u}}^{\,\,2}\,-$\\
$\,\,\,\,\,\,\,\,\,\,\,\,\,\,\,\,\,\,\,\,\,\,\,\,\,\,\,\,30\,\overline{%
\overline{u}}\,\,\overline{\overline{\overline{u}}}\,\,-5\,\overline{u}\,\,%
\overline{\overline{\overline{\overline{u}}}}$
\end{tabular}}%
\\
\hline
$2725:$&%
{\setlength{\tabcolsep}{0pt}\begin{tabular}{l}
$\overline{u^6}\,=\,1024\,\overline{u}^{\,\,6}-2304\,\overline{u}^{\,\,4}\,%
\overline{\overline{u}}\,\,+384\,\overline{u}^{\,3}\,\,\overline{\overline{%
\overline{u}}}\,+1296\,\overline{u}^{\,2}\,\,\overline{\overline{u}}^{\,\,2}\,%
-$\\
$\,\,\,\,\,\,\,\,\,\,\,\,\,\,\,\,\,\,\,\,\,\,\,288\,\overline{u}\,\,%
\overline{\overline{u}}\,\,\overline{\overline{\overline{u}}}\,\,-108\,%
\overline{\overline{u}}^{\,\,3}-24\,\overline{u}^{\,\,2}\,\,\overline{%
\overline{\overline{\overline{u}}}}\,\,+9\,\overline{\overline{u}}\,\,%
\overline{\overline{\overline{\overline{u}}}}\,\,+12\,\overline{\overline{%
\overline{u}}}^{\,\,2}$
\end{tabular}}%
\\
\hline
$14400:$&%
{\setlength{\tabcolsep}{0pt}\begin{tabular}{l}
$\overline{u^7}\,=\,4096\,\overline{u}^{\,\,7}\,-\,10752\,\overline{u}^{\,\,5%
\,}\,\overline{\overline{u}}\,+1792\,\overline{u}^{\,\,4}\,\overline{%
\overline{\overline{u}}}\,+$\\
$\,\,\,\,\,\,\,\,\,\,\,\,\,\,\,\,\,\,\,112\cdot
72\,\overline{u}^{\,\,3}\,\overline{\overline{u}}^{\,\,2}-\,2016\,%
\overline{u}^{\,\,2}\,\,\overline{\overline{u}}\,\,\overline{\overline{%
\overline{u}}}\,-1512\,\overline{u}\,\,\overline{\overline{u}}^{\,\,3}+\,$\\
$\,\,\,\,\,\,\,\,\,\,\,\,\,112\,\overline{u}\,\,\overline{\overline{%
\overline{u}}}^{\,\,2}+\,252\,\overline{\overline{u}}^{\,\,2}\,\,\overline{%
\overline{\overline{u}}}-112\,\overline{u}^{\,\,3}\,\overline{\overline{%
\overline{\overline{u}}}}\,\,+\,84\,\overline{u}\,\overline{\overline{u}}\,%
\overline{\overline{\overline{\overline{u}}}}\,\,-7\,\overline{\overline{\overline{u}}}\,%
\,\overline{\overline{\overline{\overline{u}}}}$
\end{tabular}}%
\\
\hline
\end{tabular}%
\\[1\baselineskip]
\textbf{Table \ $\overline{\text{GW}}\boldsymbol{.}\boldsymbol{\mathfrak{n}}$=5} \\[-0.5\baselineskip]
\begin{tabular}[t]{|l|l|}
\hline
{\setlength{\tabcolsep}{0pt}\begin{tabular}{l}
Sum of Positive\\
Coefficients
\end{tabular}}%
&Data Family \
$\mathcal{F}_{_5}\equiv\mathcal{V}=\left\{v_1,\,v_2,\,v_3,\,v_4,\,v_5\right%
\}$\\
\hline
$1:$&$\overline{v^1}\,=\,\overline{v}^{\,\,1}$\\
\hline
$5:$&$\overline{v^2}\,=\,5\,\,\overline{v}^{\,\,2}-4\,\overline{%
\overline{v}}$\\
\hline
$31:$&$\overline{v^3}\,=\,25\,\,\overline{v}^{\,\,3}-30\,\overline{v}\,\,%
\overline{\overline{v}}\,+6\,\overline{\overline{\overline{v}}}$\\
\hline
$205:$&$\overline{v^4}\,=\,125\,\overline{v}^{\,\,4}-200\,\overline{v}^{\,2}\,%
\overline{\overline{v}}\,+\,40\,\overline{v}\,\,\overline{\overline{%
\overline{v}}}\,\,+\,40\,\overline{\overline{v}}^{\,\,2}\,-4\,\overline{%
\overline{\overline{\overline{v}}}}$\\
\hline
$1376:$&%
{\setlength{\tabcolsep}{0pt}\begin{tabular}{l}
$\overline{v^5}\,=\,625\,\overline{v}^{\,\,5}-1250\,\overline{v}^{\,\,3}\,%
\overline{\overline{v}}\,\,+250\,\overline{v}^{\,\,2}\,\overline{\overline{%
\overline{v}}}\,\,+500\,\overline{v}\,\,\overline{\overline{v}}^{\,\,2}\,-$\\
$\,\,\,\,\,\,\,\,\,\,\,\,\,\,\,\,\,\,\,\,\,\,\,\,\,\,\,\,\,\,\,\,\,\,\,\,\,\,%
\,\,\,\,\,\,\,\,\,\,\,\,\,\,\,\,\,\,\,\,\,\,\,\,100\,\overline{\overline{v}}\,%
\,\overline{\overline{\overline{v}}}\,\,-25\,\overline{v}\,\,\overline{%
\overline{\overline{\overline{v}}}}\,\,+\,1\,\overline{\overline{\overline{%
\overline{\overline{v}}}}}$
\end{tabular}}%
$\,$\\
\hline
$9251:$&%
{\setlength{\tabcolsep}{0pt}\begin{tabular}{l}
$\overline{v^6}\,=\,3125\,\overline{v}^{\,\,6}-7500\,\overline{v}^{\,\,4}\,%
\overline{\overline{v}}\,\,+1500\,\overline{v}^{\,3}\,\,\overline{\overline{%
\overline{v}}}\,+$\\
$\,\,\,\,\,\,\,\,\,\,\,\,\,\,\,\,\,\,\,\,\,\,\,\,\,\,\,\,\,\,\,\,4500\,\,%
\overline{v}^{\,2}\,\,\overline{\overline{v}}^{\,\,2}\,-1200\,\overline{v}\,\,%
\overline{\overline{v}}\,\,\overline{\overline{\overline{v}}}\,\,-400\,%
\overline{\overline{v}}^{\,\,3}-$\\
$\,\,\,\,\,\,\,\,\,\,\,\,\,\,\,\,\,\,\,\,\,\,\,\,\,\,\,\,\,\,\,\,150\,%
\overline{v}^{\,\,2}\,\,\overline{\overline{\overline{\overline{v}}}}\,\,+60\,%
\overline{\overline{v}}\,\,\overline{\overline{\overline{\overline{v}}}}\,\,%
+60\,\overline{\overline{\overline{v}}}^{\,\,2}\,+6\,\overline{v}\,\overline{%
\overline{\overline{\overline{\overline{v}}}}}$
\end{tabular}}%
\\
\hline
$62210:$&%
{\setlength{\tabcolsep}{0pt}\begin{tabular}{l}
$\overline{v^7}\,=\,15625\,\overline{v}^{\,\,7}\,-\,43750\,\overline{v}^{\,\,%
5\,}\,\overline{\overline{v}}\,+8750\,\overline{v}^{\,\,4}\,\overline{%
\overline{\overline{v}}}\,+35000\,\overline{v}^{\,\,3}\,\overline{%
\overline{v}}^{\,\,2}-\,$\\
$\,\,\,\,\,\,\,\,\,\,\,\,\,\,10,500\,\overline{v}^{\,\,2}\,\,\overline{%
\overline{v}}\,\,\overline{\overline{\overline{v}}}\,-\,7000\,\overline{v}\,\,%
\overline{\overline{v}}^{\,\,3}+700\,\overline{v}\,\,\overline{\overline{%
\overline{v}}}^{\,\,2}+\,1400\,\overline{\overline{v}}^{\,\,2}\,\,\overline{%
\overline{\overline{v}}}-$\\
$\,\,\,\,\,\,\,\,\,\,\,\,\,\,\,\,\,875\,\overline{v}^{\,\,3}\,\overline{%
\overline{\overline{\overline{v}}}}\,\,+\,700\,\overline{v}\,\overline{%
\overline{v}}\,\overline{\overline{\overline{\overline{v}}}}\,\,-70\,%
\overline{\overline{\overline{v}}}\,\,\overline{\overline{\overline{%
\overline{v}}}}\,\,+35\,\overline{v}^{\,\,2}\,\overline{\overline{\overline{%
\overline{\overline{v}}}}}\,\,-\,14\,\overline{\overline{v}}\,\,\overline{%
\overline{\overline{\overline{\overline{v}}}}}$
\end{tabular}}%
\\
\hline
\end{tabular}%
\\[1\baselineskip]
\newpage
\textbf{Table \ $\overline{\text{GW}}\boldsymbol{.}\boldsymbol{\mathfrak{n}}$=6} \\[-0.5\baselineskip]
\begin{tabular}[t]{|l|l|}
\hline
{\setlength{\tabcolsep}{0pt}\begin{tabular}{l}
Sum of Positive\\
Coefficients
\end{tabular}}%
&Data Family \
$\mathcal{F}_{_6}\equiv\mathcal{W}=\left\{w_1,\,w_2,\,w_3,\,w_4,\,w_5,\,w_6%
\right\}$\\
\hline
$1:$&$\overline{w^1}\,=\,\overline{w}^{\,\,1}$\\
\hline
$6:$&$\overline{w^2}\,=\,6\,\,\overline{w}^{\,\,2}-5\,\overline{%
\overline{w}}$\\
\hline
$46:$&$\overline{w^3}\,=\,36\,\,\overline{w}^{\,\,3}-45\,\overline{w}\,\,%
\overline{\overline{w}}\,+10\,\overline{\overline{\overline{w}}}$\\
\hline
$371:$&$\overline{w^4}\,=\,216\,\overline{w}^{\,\,4}-360\,\overline{w}^{\,2}\,%
\overline{\overline{w}}\,+\,80\,\overline{w}\,\,\overline{\overline{%
\overline{w}}}\,\,+\,75\,\overline{\overline{w}}^{\,\,2}\,-10\,\overline{%
\overline{\overline{\overline{w}}}}$\\
\hline
$3026:$&%
{\setlength{\tabcolsep}{0pt}\begin{tabular}{l}
$\,\overline{w^5}\,=\,1296\,\overline{w}^{\,\,5}-2700\,\overline{w}^{\,\,3}\,%
\overline{\overline{w}}\,\,+600\,\overline{w}^{\,\,2}\,\overline{\overline{%
\overline{w}}}\,\,+1125\,\overline{w}\,\,\overline{\overline{w}}^{\,\,2}\,-$\\
$\,\,\,\,\,\,\,\,\,\,\,\,\,\,\,\,\,\,\,\,\,\,\,\,\,\,\,\,\,\,\,\,\,\,\,\,\,\,%
\,\,\,\,\,\,\,\,\,\,\,\,\,\,250\,\overline{\overline{w}}\,\,\overline{%
\overline{\overline{w}}}\,\,-75\,\overline{w}\,\,\overline{\overline{%
\overline{\overline{w}}}}\,\,+\,5\,\overline{\overline{\overline{\overline{%
\overline{w}}}}}$
\end{tabular}}%
\\
\hline
$24707:$&%
{\setlength{\tabcolsep}{0pt}\begin{tabular}{l}
$\overline{w^6}\,=\,7776\,\overline{w}^{\,\,6}-19440\,\overline{w}^{\,\,4}\,%
\overline{\overline{w}}\,\,+4320\,\overline{w}^{\,3}\,\,\overline{\overline{%
\overline{w}}}\,+$\\
$\,\,\,\,\,\,\,\,\,\,\,\,\,\,\,\,\,\,12150\,\,\overline{w}^{\,2}\,\,%
\overline{\overline{w}}^{\,\,2}\,-3600\,\overline{w}\,\,\overline{%
\overline{w}}\,\,\overline{\overline{\overline{w}}}\,\,-\,1125\,\overline{%
\overline{w}}^{\,\,3}-$\\
$\,\,\,\,\,\,\,\,\,\,\,\,540\,\overline{w}^{\,\,2}\,\,\overline{\overline{%
\overline{\overline{w}}}}\,\,+225\,\overline{\overline{w}}\,\,\overline{%
\overline{\overline{\overline{w}}}}\,\,+200\,\overline{\overline{%
\overline{w}}}^{\,\,2}+36\,\overline{w}\,\overline{\overline{\overline{%
\overline{\overline{w}}}}}\,\,-\,1\,\overline{\overline{\overline{\overline{%
\overline{\overline{w}}}}}}$
\end{tabular}}%
\\
\hline
$201748:$&%
{\setlength{\tabcolsep}{0pt}\begin{tabular}{l}
$\overline{w^7}\,=\,46656\,\overline{w}^{\,\,7}\,-\,136080\,\overline{w}^{\,\,%
5\,}\,\overline{\overline{w}}\,+30240\,\overline{w}^{\,\,4}\,\overline{%
\overline{\overline{w}}}\,+$\\
$\,\,\,\,\,\,\,\,\,\,\,\,\,\,\,\,113400\,\overline{w}^{\,\,3}\,\overline{%
\overline{w}}^{\,\,2}-\,37,800\,\overline{w}^{\,\,2}\,\,\overline{%
\overline{w}}\,\,\overline{\overline{\overline{w}}}\,-\,23625\,\overline{w}\,%
\,\overline{\overline{w}}^{\,\,3}+$\\
$\,\,\,\,\,\,\,\,\,\,\,\,\,\,2800\,\overline{w}\,\,\overline{\overline{%
\overline{w}}}^{\,\,2}+\,5250\,\overline{\overline{w}}^{\,\,2}\,\,\overline{%
\overline{\overline{w}}}-3780\,\overline{w}^{\,\,3}\,\overline{\overline{%
\overline{\overline{w}}}}\,+$\\
$\,\,\,\,\,\,\,\,\,\,\,3150\,\overline{w}\,\overline{\overline{w}}\,%
\overline{\overline{\overline{\overline{w}}}}\,\,-350\,\overline{\overline{%
\overline{w}}}\,\,\overline{\overline{\overline{\overline{w}}}}\,\,+252\,%
\overline{w}^{\,\,2}\,\overline{\overline{\overline{\overline{%
\overline{w}}}}}\,\,-\,105\,\overline{\overline{w}}\,\,\overline{\overline{%
\overline{\overline{\overline{w}}}}}\,\,-\,7\,\overline{\overline{w}}\,\,%
\overline{\overline{\overline{\overline{\overline{w}}}}}$
\end{tabular}}%
\\
\hline
\end{tabular}%
\\[3\baselineskip]
Collated, instead, along powers, the same information is:\\[1\baselineskip]
\noindent
\textbf{Table \ $\overline{\text{GW}}\boldsymbol{.}$deg=2}\\[-0.5\baselineskip]
\begin{tabular}[t]{|l|l|l|}
\hline
$\left|\mathcal{F}_{_n}\right|$&%
{\setlength{\tabcolsep}{0pt}\begin{tabular}{l}
Sum of Positive\\
Coefficients
\end{tabular}}%
&Quadratic\\
\hline
2&2&$\overline{s^2}\,=\,2\,\,\overline{s}^{\,\,2}-1\,\overline{\overline{s}}$%
\\
\hline
3&3&$\overline{t^2}\,=\,3\,\,\overline{t}^{\,\,2}-2\,\overline{\overline{t}}$%
\\
\hline
4&4&$\overline{u^2}\,=\,4\,\,\overline{u}^{\,\,2}-3\,\overline{\overline{u}}$%
\\
\hline
5&5&$\overline{v^2}\,=\,5\,\,\overline{v}^{\,\,2}-4\,\overline{\overline{v}}$%
\\
\hline
$6$&$6$&$\overline{w^2}\,=\,6\,\,\overline{w}^{\,\,2}-5\,\overline{%
\overline{w}}$\\
\hline
\end{tabular}%
\\[1\baselineskip]
\newpage
\noindent
\textbf{Table \ $\overline{\text{GW}}\boldsymbol{.}$deg=3}\\[-0.5\baselineskip]
\begin{tabular}[t]{|l|l|l|}
\hline
$\left|\mathcal{F}_{_n}\right|$&%
{\setlength{\tabcolsep}{0pt}\begin{tabular}{l}
Sum of $\left(+\right)$\\
Coefficients
\end{tabular}}%
&Cubic\\
\hline
2&$4$&$\overline{s^3}\,=\,4\,\,\overline{s}^{\,\,3}-3\,\overline{s}\,\,%
\overline{\overline{s}}$\\
\hline
3&$10$&$\overline{t^3}\,=\,9\,\,\overline{t}^{\,\,3}-9\,\overline{t}\,\,%
\overline{\overline{t}}\,+\overline{\overline{\overline{t}}}$\\
\hline
4&$19$&$\overline{u^3}\,=\,16\,\,\overline{u}^{\,\,3}-18\,\overline{u}\,\,%
\overline{\overline{u}}\,+3\,\overline{\overline{\overline{u}}}$\\
\hline
5&$31$&$\overline{v^3}\,=\,25\,\,\overline{v}^{\,\,3}-30\,\overline{v}\,\,%
\overline{\overline{v}}\,+6\,\overline{\overline{\overline{v}}}$\\
\hline
$6$&$46$&$\overline{w^3}\,=\,36\,\,\overline{w}^{\,\,3}-45\,\overline{w}\,\,%
\overline{\overline{w}}\,+10\,\overline{\overline{\overline{w}}}$\\
\hline
\end{tabular}%
\\[1\baselineskip]
\noindent
\textbf{Table \ $\overline{\text{GW}}\boldsymbol{.}$deg=4}\\[-0.5\baselineskip]
\begin{tabular}[t]{|l|l|l|}
\hline
$\left|\mathcal{F}_{_n}\right|$&%
{\setlength{\tabcolsep}{0pt}\begin{tabular}{l}
Sum of $\left(+\right)$\\
Coefficients
\end{tabular}}%
&Quartic\\
\hline
2&$9$&$\overline{s^4}\,=\,8\,\overline{s}^{\,\,4}-8\,\overline{s}^{\,2}\,%
\overline{\overline{s}}\,+1\,\overline{\overline{s}}^{\,\,2}$\\
\hline
3&$37$&$\overline{t^4}\,=\,27\,\overline{t}^{\,\,4}-36\,\overline{t}^{\,2}\,%
\overline{\overline{t}}\,+\,4\,\overline{t}\,\overline{\overline{%
\overline{t}}}\,\,+\,6\,\overline{\overline{t}}^{\,\,2}$\\
\hline
4&$98$&$\overline{u^4}\,=\,64\,\overline{u}^{\,\,4}-96\,\overline{u}^{\,2}\,%
\overline{\overline{u}}\,+\,16\,\overline{u}\,\overline{\overline{%
\overline{u}}}\,\,+\,18\,\overline{\overline{u}}^{\,\,2}\,-\,\overline{%
\overline{\overline{\overline{u}}}}$\\
\hline
5&$205$&$\overline{v^4}\,=\,125\,\overline{v}^{\,\,4}-200\,\overline{v}^{\,2}%
\,\overline{\overline{v}}\,+\,40\,\overline{v}\,\,\overline{\overline{%
\overline{v}}}\,\,+\,40\,\overline{\overline{v}}^{\,\,2}\,-4\,\overline{%
\overline{\overline{\overline{v}}}}$\\
\hline
$6$&$371$&$\overline{w^4}\,=\,216\,\overline{w}^{\,\,4}-360\,\overline{w}^{\,%
2}\,\overline{\overline{w}}\,+\,80\,\overline{w}\,\,\overline{\overline{%
\overline{w}}}\,\,+\,75\,\overline{\overline{w}}^{\,\,2}\,-10\,\overline{%
\overline{\overline{\overline{w}}}}$\\
\hline
\end{tabular}%
\\[1\baselineskip]
\noindent
\textbf{Table \ $\overline{\text{GW}}\boldsymbol{.}$deg=5}\\[-0.5\baselineskip]
\begin{tabular}[t]{|l|l|l|}
\hline
$\left|\mathcal{F}_{_n}\right|$&%
{\setlength{\tabcolsep}{0pt}\begin{tabular}{l}
Sum of $\left(+\right)$\\
Coefficients
\end{tabular}}%
&Quintic\\
\hline
2&$21$&$\overline{s^5}\,=\,16\,\overline{s}^{\,\,5}-20\,\overline{s}^{\,\,3}\,%
\overline{\overline{s}}\,+5\,\overline{s}\,\overline{\overline{s}}^{\,\,2}$\\
\hline
3&$141$&$\overline{t^5}\,=\,81\,\overline{t}^{\,\,5}-135\,\overline{t}^{\,\,%
3}\,\overline{\overline{t}}\,\,+15\,\overline{t}^{\,\,2}\,\overline{%
\overline{\overline{t}}}\,\,+45\,\overline{t}\,\overline{\overline{t}}^{\,\,%
2}\,-5\,\overline{\overline{t}}\,\overline{\overline{\overline{t}}}$\\
\hline
4&$516$&%
{\setlength{\tabcolsep}{0pt}\begin{tabular}{l}
$\overline{u^5}\,=\,256\,\overline{u}^{\,\,5}-480\,\overline{u}^{\,\,3}\,%
\overline{\overline{u}}\,\,+80\,\overline{u}^{\,\,2}\,\overline{\overline{%
\overline{u}}}\,\,+$\\
$\,\,\,\,\,\,\,\,\,\,\,\,\,\,\,\,\,\,\,\,\,\,\,\,\,\,180\,\overline{u}\,%
\overline{\overline{u}}^{\,\,2}-30\,\overline{\overline{u}}\,\,\overline{%
\overline{\overline{u}}}\,\,-5\,\overline{u}\,\,\overline{\overline{%
\overline{\overline{u}}}}$
\end{tabular}}%
\\
\hline
5&$1376$&%
{\setlength{\tabcolsep}{0pt}\begin{tabular}{l}
$\overline{v^5}\,=\,625\,\overline{v}^{\,\,5}-1250\,\overline{v}^{\,\,3}\,%
\overline{\overline{v}}\,\,+250\,\overline{v}^{\,\,2}\,\overline{\overline{%
\overline{v}}}\,\,+$\\
$\,\,\,\,\,\,\,\,\,\,\,\,\,\,\,\,\,500\,\overline{v}\,\,\overline{%
\overline{v}}^{\,\,2}-100\,\overline{\overline{v}}\,\,\overline{\overline{%
\overline{v}}}\,\,-25\,\overline{v}\,\,\overline{\overline{\overline{%
\overline{v}}}}\,\,+\,1\,\overline{\overline{\overline{\overline{%
\overline{v}}}}}$
\end{tabular}}%
\\
\hline
$6$&$3026$&%
{\setlength{\tabcolsep}{0pt}\begin{tabular}{l}
$\,\overline{w^5}\,=\,1296\,\overline{w}^{\,\,5}-2700\,\overline{w}^{\,\,3}\,%
\overline{\overline{w}}\,\,+600\,\overline{w}^{\,\,2}\,\overline{\overline{%
\overline{w}}}\,\,+$\\
$\,\,\,\,\,\,\,\,\,\,\,\,\,\,\,1125\,\overline{w}\,\,\overline{%
\overline{w}}^{\,\,2}-250\,\overline{\overline{w}}\,\,\overline{\overline{%
\overline{w}}}\,\,-75\,\overline{w}\,\,\overline{\overline{\overline{%
\overline{w}}}}\,\,+\,5\,\overline{\overline{\overline{\overline{%
\overline{w}}}}}$
\end{tabular}}%
\\
\hline
\end{tabular}%
\\[1\baselineskip]
\newpage
\noindent
\textbf{Table \ $\overline{\text{GW}}\boldsymbol{.}$deg=6}\\[-0.5\baselineskip]
\begin{tabular}[t]{|l|l|l|}
\hline
$\left|\mathcal{F}_{_n}\right|$&%
{\setlength{\tabcolsep}{0pt}\begin{tabular}{l}
Sum of $\left(+\right)$\\
Coefficients
\end{tabular}}%
&Sextic\\
\hline
2&$50$&$\overline{s^6}\,=\,32\,\overline{s}^{\,\,6}-48\,\overline{s}^{\,\,4}\,%
\overline{\overline{s}}\,\,+18\,\overline{s}^{\,2}\,\,\overline{%
\overline{s}}^{\,\,2}\,-\overline{\overline{s}}^{\,3}$\\
\hline
3&$541$&%
{\setlength{\tabcolsep}{0pt}\begin{tabular}{l}
$\overline{t^6}\,=\,243\,\overline{t}^{\,\,6}-486\,\overline{t}^{\,\,4}\,%
\overline{\overline{t}}\,\,+54\,\overline{t}^{\,3}\,\,\overline{\overline{%
\overline{t}}}\,-$\\
$\,\,\,\,\,\,\,\,\,\,\,\,\,\,\,\,243\,\overline{t}^{\,2}\,\,\overline{%
\overline{t}}^{\,\,2}-36\,\overline{t}\,\,\overline{\overline{t}}\,\,%
\overline{\overline{\overline{t}}}\,\,-18\,\overline{\overline{t}}^{\,\,3}$
\end{tabular}}%
$\,\,$\\
\hline
4&$2725$&%
{\setlength{\tabcolsep}{0pt}\begin{tabular}{l}
$\overline{u^6}\,=\,1024\,\overline{u}^{\,\,6}-2304\,\overline{u}^{\,\,4}\,%
\overline{\overline{u}}\,\,+384\,\overline{u}^{\,3}\,\,\overline{\overline{%
\overline{u}}}\,+$\\
$\,\,\,\,\,\,\,\,\,\,\,\,\,\,\,\,\,\,\,\,\,\,1296\,\overline{u}^{\,2}\,\,%
\overline{\overline{u}}^{\,\,2}\,-288\,\overline{u}\,\,\overline{%
\overline{u}}\,\,\overline{\overline{\overline{u}}}\,\,-$\\
$\,\,\,\,\,\,\,\,\,\,\,\,108\,\overline{\overline{u}}^{\,\,3}-24\,%
\overline{u}^{\,\,2}\,\,\overline{\overline{\overline{\overline{u}}}}\,\,+9\,%
\overline{\overline{u}}\,\,\overline{\overline{\overline{\overline{u}}}}\,\,%
+12\,\overline{\overline{\overline{u}}}^{\,\,2}$
\end{tabular}}%
\\
\hline
5&9251&%
{\setlength{\tabcolsep}{0pt}\begin{tabular}{l}
$\overline{v^6}\,=\,3125\,\overline{v}^{\,\,6}-7500\,\overline{v}^{\,\,4}\,%
\overline{\overline{v}}\,\,+1500\,\overline{v}^{\,3}\,\,\overline{\overline{%
\overline{v}}}\,+$\\
$\,\,\,\,\,\,\,\,\,\,\,\,\,\,4500\,\,\overline{v}^{\,2}\,\,\overline{%
\overline{v}}^{\,\,2}\,-1200\,\overline{v}\,\,\overline{\overline{v}}\,\,%
\overline{\overline{\overline{v}}}\,\,-400\,\overline{\overline{v}}^{\,\,3}-$%
\\
$\,\,\,\,\,\,\,\,\,\,\,\,\,\,\,150\,\overline{v}^{\,\,2}\,\,\overline{%
\overline{\overline{\overline{v}}}}\,\,+60\,\overline{\overline{v}}\,\,%
\overline{\overline{\overline{\overline{v}}}}\,\,+60\,\overline{\overline{%
\overline{v}}}^{\,\,2}\,+6\,\overline{v}\,\overline{\overline{\overline{%
\overline{\overline{v}}}}}$
\end{tabular}}%
\\
\hline
$6$&$24707$&%
{\setlength{\tabcolsep}{0pt}\begin{tabular}{l}
$\overline{w^6}\,=\,7776\,\overline{w}^{\,\,6}-19440\,\overline{w}^{\,\,4}\,%
\overline{\overline{w}}\,\,+4320\,\overline{w}^{\,3}\,\,\overline{\overline{%
\overline{w}}}\,+$\\
$\,\,\,\,\,\,\,\,\,12150\,\,\overline{w}^{\,2}\,\,\overline{\overline{w}}^{\,%
\,2}\,-3600\,\overline{w}\,\,\overline{\overline{w}}\,\,\overline{\overline{%
\overline{w}}}\,\,-\,1125\,\overline{\overline{w}}^{\,\,3}-$\\
$\,\,540\,\overline{w}^{\,\,2}\,\,\overline{\overline{\overline{%
\overline{w}}}}\,\,+225\,\overline{\overline{w}}\,\,\overline{\overline{%
\overline{\overline{w}}}}\,\,+200\,\overline{\overline{\overline{w}}}^{\,\,%
2}+36\,\overline{w}\,\overline{\overline{\overline{\overline{\overline{w}}}}}%
\,\,-\,1\,\overline{\overline{\overline{\overline{\overline{\overline{w}}}}}}$
\end{tabular}}%
\\
\hline
\end{tabular}%
\\[1\baselineskip]
\noindent
\textbf{Table \ $\overline{\text{GW}}\boldsymbol{.}$deg=7}\\[-0.5\baselineskip]
\begin{small}
\begin{tabular}[t]{|l|l|l|}
\hline
$\left|\mathcal{F}_{_n}\right|$&%
{\setlength{\tabcolsep}{0pt}\begin{tabular}{l}
Sum of $\left(+\right)$\\
Coefficients
\end{tabular}}%
&Septic\\
\hline
2&$120$&$\overline{s^7}\,=\,64\,\overline{s}^{\,\,7}\,-\,112\,\overline{s}^{\,%
\,5}\,\overline{\overline{s}}\,+56\,\overline{s}^{\,3}\,\overline{%
\overline{s}}^{\,\,2}\,-\,7\,\overline{s}\,\overline{\overline{s}}^{\,\,3}$\\
\hline
3&$2080$&%
{\setlength{\tabcolsep}{0pt}\begin{tabular}{l}
$\overline{t^7}\,=\,7294\,\overline{t}^{\,\,7}\,-\,1701\,\overline{t}^{\,\,5}%
\,\overline{\overline{t}}\,+189\,\overline{t}^{\,\,4}\,\overline{\overline{%
\overline{t}}}\,+1134\,\overline{t}^{\,\,3}\,\overline{\overline{t}}^{\,\,%
2}-$\\
$\,\,\,\,\,\,\,\,\,\,\,\,\,\,\,\,\,\,\,\,\,\,\,\,\,\,\,\,189\,\overline{t}^{\,%
\,2}\,\,\overline{\overline{t}}\,\,\overline{\overline{\overline{t}}}\,-\,189%
\,\overline{t}\,\,\overline{\overline{t}}^{\,\,3}+7\,\overline{t}\,\,%
\overline{\overline{\overline{t}}}^{\,\,2}+\,21\,\overline{\overline{t}}^{\,\,%
2}\,\,\overline{\overline{\overline{t}}}$
\end{tabular}}%
\\
\hline
4&$14400$&%
{\setlength{\tabcolsep}{0pt}\begin{tabular}{l}
$\overline{u^7}\,=\,4096\,\overline{u}^{\,\,7}\,-\,10752\,\overline{u}^{\,\,5%
\,}\,\overline{\overline{u}}\,+1792\,\overline{u}^{\,\,4}\,\overline{%
\overline{\overline{u}}}\,+$\\
$\,\,\,\,\,\,\,\,\,\,\,\,\,\,\,\,\,\,\,112\cdot
72\,\overline{u}^{\,\,3}\,\overline{\overline{u}}^{\,\,2}-\,2016\,%
\overline{u}^{\,\,2}\,\,\overline{\overline{u}}\,\,\overline{\overline{%
\overline{u}}}\,-1512\,\overline{u}\,\,\overline{\overline{u}}^{\,\,3}+\,$\\
$\,\,\,\,\,\,\,\,\,\,\,\,\,112\,\overline{u}\,\,\overline{\overline{%
\overline{u}}}^{\,\,2}+\,252\,\overline{\overline{u}}^{\,\,2}\,\,\overline{%
\overline{\overline{u}}}-112\,\overline{u}^{\,\,3}\,\overline{\overline{%
\overline{\overline{u}}}}\,\,+\,84\,\overline{u}\,\overline{\overline{u}}\,%
\overline{\overline{\overline{u}}}\,\,-7\,\overline{\overline{\overline{u}}}\,%
\,\overline{\overline{\overline{\overline{u}}}}$
\end{tabular}}%
\\
\hline
5&$62210$&%
{\setlength{\tabcolsep}{0pt}\begin{tabular}{l}
$\overline{v^7}\,=\,15625\,\overline{v}^{\,\,7}\,-\,43750\,\overline{v}^{\,\,%
5\,}\,\overline{\overline{v}}\,+8750\,\overline{v}^{\,\,4}\,\overline{%
\overline{\overline{v}}}\,+35000\,\overline{v}^{\,\,3}\,\overline{%
\overline{v}}^{\,\,2}-\,$\\
$\,\,\,\,\,\,\,\,\,\,\,\,\,\,10,500\,\overline{v}^{\,\,2}\,\,\overline{%
\overline{v}}\,\,\overline{\overline{\overline{v}}}\,-\,7000\,\overline{v}\,\,%
\overline{\overline{v}}^{\,\,3}+700\,\overline{v}\,\,\overline{\overline{%
\overline{v}}}^{\,\,2}+\,1400\,\overline{\overline{v}}^{\,\,2}\,\,\overline{%
\overline{\overline{v}}}-$\\
$\,\,\,\,\,\,\,\,\,\,\,\,\,\,\,\,\,875\,\overline{v}^{\,\,3}\,\overline{%
\overline{\overline{\overline{v}}}}\,\,+\,700\,\overline{v}\,\overline{%
\overline{v}}\,\overline{\overline{\overline{\overline{v}}}}\,\,-70\,%
\overline{\overline{\overline{v}}}\,\,\overline{\overline{\overline{%
\overline{v}}}}\,\,+35\,\overline{v}^{\,\,2}\,\overline{\overline{\overline{%
\overline{\overline{v}}}}}\,\,-\,14\,\overline{\overline{v}}\,\,\overline{%
\overline{\overline{\overline{\overline{v}}}}}$
\end{tabular}}%
\\
\hline
$6$&$201748$&%
{\setlength{\tabcolsep}{0pt}\begin{tabular}{l}
$\overline{w^7}\,=\,46656\,\overline{w}^{\,\,7}\,-\,136080\,\overline{w}^{\,\,%
5\,}\,\overline{\overline{w}}\,+30240\,\overline{w}^{\,\,4}\,\overline{%
\overline{\overline{w}}}\,+$\\
$\,\,\,\,\,\,\,\,\,\,\,\,\,\,\,\,113400\,\overline{w}^{\,\,3}\,\overline{%
\overline{w}}^{\,\,2}-\,37,800\,\overline{w}^{\,\,2}\,\,\overline{%
\overline{w}}\,\,\overline{\overline{\overline{w}}}\,-\,23625\,\overline{w}\,%
\,\overline{\overline{w}}^{\,\,3}+$\\
$\,\,\,\,\,\,\,\,\,\,\,\,\,\,2800\,\overline{w}\,\,\overline{\overline{%
\overline{w}}}^{\,\,2}+\,5250\,\overline{\overline{w}}^{\,\,2}\,\,\overline{%
\overline{\overline{w}}}-3780\,\overline{w}^{\,\,3}\,\overline{\overline{%
\overline{\overline{w}}}}\,+$\\
$\,\,\,\,\,\,\,\,\,\,\,3150\,\overline{w}\,\overline{\overline{w}}\,%
\overline{\overline{\overline{\overline{w}}}}\,\,-350\,\overline{\overline{%
\overline{w}}}\,\,\overline{\overline{\overline{\overline{w}}}}\,\,+252\,%
\overline{w}^{\,\,2}\,\overline{\overline{\overline{\overline{%
\overline{w}}}}}\,\,-\,105\,\overline{\overline{w}}\,\,\overline{\overline{%
\overline{\overline{\overline{w}}}}}\,\,-\,7\,\overline{\overline{w}}\,\,%
\overline{\overline{\overline{\overline{\overline{w}}}}}$
\end{tabular}}%
\\
\hline
\end{tabular}%
\end{small}

\begin{acknowledgment}{Acknowledgments}
thanks to all...
\end{acknowledgment}

\begin{biog}
\item[first author]first author
\begin{affil}
\end{affil}

\item[2nd author]second author
\begin{affil}
\end{affil}

\item[3rd author]third author
\begin{affil}
\end{affil}
\end{biog}

\end{document}